\documentclass[a4paper,10pt]{amsart}

\usepackage{amsmath}
\usepackage{amssymb}
\usepackage{amsthm}
\usepackage{verbatim}
\usepackage[all]{xy}
\usepackage{enumerate}

\usepackage{color}

\newtheorem{thm}{Theorem}[section]
\newtheorem{prop}[thm]{Proposition}
\newtheorem{cor}[thm]{Corollary}
\newtheorem{lem}[thm]{Lemma}

\theoremstyle{definition}

\newtheorem{dfn}[thm]{Definition}

\newtheorem{thmdfn}{Theorem-Definition}[section]

\newtheorem{rmk}[thm]{Remark}

\numberwithin{equation}{section}
\newcommand{\spa}{\textrm{span}}

\newcommand{\cH}{\mathcal{H}}

\newcommand{\bG}{\mathbb{G}}
\newcommand{\MCB}{M_0A}
\newcommand{\SUqone}{SU_q(1,1)_{{\rm ext}}}
\newcommand{\ep}{\epsilon}
\newcommand{\sgn}{{\rm sgn}}
\newcommand{\SUone}{SU_q(1,1)_{{\rm ext}}}
\newcommand{\CB}{\mathcal{CB}}
\newcommand{\cG}{\mathcal{G}}

\newcommand{\cI}{\mathcal{I}}
\newcommand{\cT}{\mathcal{T}}
\newcommand{\id}{\textrm{id}}
\title[Weak amenability of locally compact quantum groups]{Weak amenability of locally compact quantum groups and approximation properties of extended quantum $SU(1,1)$}

\date{\noindent \today.  \\
The  author was supported by the ANR project: ANR-2011-BS01-008-01. 
 }  

\author{Martijn Caspers}
 \address{M. Caspers, Laboratoire de Math\'ematiques, Universit\'e de Franche-Comt\'e, 16 Route de Gray, 25030 Besan\c con, France}
 \email{martijn.caspers@univ-fcomte.fr}

\begin{document}
 
\maketitle

\begin{abstract}
We study weak amenability for locally compact quantum groups in the sense of Kustermans and Vaes. In particular, we focus on non-discrete examples. We prove that a coamenable quantum group is weakly amenable if there exists a net of positive, scaling invariant elements in the Fourier algebra $A(\bG)$ whose representing multipliers form an  approximate identity in $C_0(\bG)$ that is bounded in the $\MCB(\bG)$ norm; the bound being an upper estimate for the associated Cowling-Haagerup constant.  

As an application, we find the appropriate approximation properties of the extended quantum $SU(1,1)$ group and its dual. That is, we prove that it is weakly amenable and coamenable. Furthermore, it has the Haagerup property in the quantum group sense,  introduced  by Daws, Fima, Skalski and White. 
\end{abstract}

\section{Introduction}
 
Locally compact quantum groups have been introduced by Kustermans and Vaes in their papers \cite{KusVae}, \cite{KusVaeII}. They form a category larger than locally compact groups, admitting a full Pontrjagin duality theorem.

A class of examples occurs as deformations of the algebra of continuous or measurable functions on a locally compact group, for example the deformations of simple compact Lie groups. Furthermore, suitable deformations of $E(2)$ and $SU(1,1)$ have been constructed. 

These examples give rise to the question which approximation properties these quantum groups and their duals have. For the deformation of $E(2)$ (see \cite{Wor}, \cite{Jacobs}) it is known that it is both {\it amenable} and {\it coamenable}. Here, amenability is the (or at least one of the) quantum generalization(s) of group amenability, whereas coamenability is the dual notion. Deformations of  compact simple Lie groups are coamenable (c.f. \cite[Corollary 6.2]{Banica}) and trivially amenable.

Recall that {\it coamenability} asserts the existence of a bounded approximate identity in the convolution algebra $L^1(\bG)$ of a locally compact quantum group $\bG$. {\it Amenability} can for example be characterized by the fact that the universal and reduced dual quantum groups $C^\ast_u(\bG)$ and $C^\ast_r(\bG)$ are equal. 

Amenability and coamenability of a quantum group are dual to each other, in the sense that for a compact quantum group it is known that it is coamenable if and only if the Pontrjagin dual quantum group is amenable, see the work of Ruan \cite{RuanAm} and Tomatsu \cite{Tomatsu}. The question whether or not this generalizes to arbitrary quantum groups remains open. 

For groups subsequently weaker approximation properties than amenability have been introduced, in particular {\it weak amenability} \cite{CanHaa}. These notions have been generalized to quantum groups of Kac type by Kraus and Ruan \cite{Ruan}. In particular, equivalent intrinsic approximation properties of the C$^\ast$-algebra and von Neumann algebra of a discrete Kac algebra have been found. For example weak amenability of $\mathbb{G}$ corresponds to the completely bounded approximation property (CBAP) of $C^\ast_r(\mathbb{G})$.

Let us mention that in  \cite{Freslon} Freslon proved that the free orthogonal and free unitary quantum groups are weakly amenable. The result was extended to the non-Kac deformations in \cite{ComFreYam} in terms of the CBAP. The current paper deals with a weak amenability result for non-discrete quantum groups, which is of different flavor, basically since it does not have an interpretation in terms of the CBAP. 

Various equivalent notions of the Haagerup property for locally compact quantum groups were recently investigated by Daws, Fima, Skalski and White \cite{DawFimSkaWhi}. Examples so far come from Brannan's result \cite{Bran} on free orthogonal and free unitary quantum groups  and Lemeux \cite{Lem} for quantum reflexion groups. Beyond the discrete case, the only known (non-classical) examples  are amenable or follow from coamenability of the dual. 

\vspace{0.3cm}

The $\SUone$ group was first established as a proper von Neumann algebraic quantum group by Koelink and Kustermans \cite{KoeKus}. One of the main difficulties was to prove the coassociativity of the comultiplication. More recently a novel way of obtaining $\SUone$ was found by De Commer \cite{ComII} (see also \cite{ComI}),   avoiding the proof of the coassociativity. In \cite{GroKoeKus} Groenevelt, Koelink and Kustermans obtain  the Plancherel decomposition of the multiplicative unitary of $\SUone$. Part of this paper is  based on this achievement.  We expand below in more detail.

Let us mention that the construction of non-compact operator algebraic quantum groups is one of the major open questions.  However, there is reasonable hope that the techniques of \cite{ComI} allow a passage to a larger class of examples. In fact, in \cite{ComIII} it was shown that also a version of  $E(2)$ can be recovered. 

\vspace{0.3cm}

We summarize the main results of this paper. Section \ref{Sect=CompactaApproximation} provides a sufficient condition for a quantum group to be weakly amenable. Explaining the notation and definitions in the subsequent sections, we state:

\vspace{0.3cm}

{\bf Theorem 1}. Let $\bG$ be a coamenable locally compact quantum group. Suppose that there exists a net $\{ a_i \}$ of positive elements in the Fourier algebra $A(\bG)$, whose representing elements in $C_0(\bG)$ are invariant under the scaling group and such that they form an approximate identity for $C_0(\bG)$ with   $\Vert a_i \Vert_{\MCB(\bG)}$  bounded. Then, $\bG$ is weakly amenable.

\vspace{0.3cm}

The theorem is based on the arguments of De Canniere and Haagerup \cite{CanHaa} and is typically applicable if the trivial corepresentation lies nicely in the representation spectrum of a quantum group. Compared to \cite{CanHaa}, besides technical difficulties,  we encounter two new phenomena in the quantum case.  It appears that  for quantum groups   {\it coamenability} plays an important role. Also, it turns out that we need to work with {\it modular  } multipliers (we state the definition later). Recall that classical groups are always coamenable and the modular assumption is trivially satisfied. 

\vspace{0.3cm}

As an example, we find the appropriate approximation properties of $\SUone$ and its dual. Our proofs rely on $q$-analysis of little $q$-Jacobi functions that are special cases of $_2 \varphi _1$-hypergeometric series. As a first result, we find:

\vspace{0.3cm}

\noindent {\bf Theorem A.} $\SUone$ is coamenable.

\vspace{0.3cm}

Let us mention that coamenability was recently used in relation with idempotent states on locally compact quantum groups \cite{PekSal}. It turns out that $\SUone$ falls within the category of examples, answering a question in \cite[Section 3]{PekSal}.

As a consequence of {\bf Theorem A}, we are able to prove the following.

\vspace{0.3cm}

\noindent {\bf Theorem B.} $\SUone$ is weakly amenable with Cowling-Haagerup constant 1. 

\vspace{0.3cm}

 The proof of {\bf Theorem B} relies on computations based on the Plancherel decomposition \cite{GroKoeKus} and then follows the proof of Haagerup and De Canniere \cite{CanHaa} in order to apply {\bf Theorem 1}. As in the classical case (see also \cite[Remark 3.8]{CanHaa}), we believe that our proofs are adaptable to deformations of Lie groups of which the identity operator lies nicely in the spectrum of corepresentations.  Finally, we find that:

\vspace{0.3cm}

\noindent {\bf Theorem C.} $\SUone$ has the Haagerup property.

\vspace{0.3cm}

These are the first (genuine) examples of non-discrete, non-amenable quantum groups that: are weakly amenable/have the Haagerup property. Also the dual of $\SUone$ has the Haagerup property, as will follow from {\bf Theorem A}. 

\vspace{0.3cm}

The structure of this paper is as follows. Section \ref{Sect=Lcqg} recalls the definition of a locally compact quantum group. In Section \ref{Sect=CompactaApproximation} we recall the definition of amenability and prove   {\bf Theorem 1}. In Section \ref{Sect=SUone} we recall the necessary preliminaries on $\SUqone$. Section \ref{Sect=Coamenability} proves {\bf Theorem A}. Sections \ref{Sect=Spherical} and  \ref{Sect=WeakAmenability} prove {\bf Theorem B}. Section \ref{Sect=Haagerup} proves {\bf Theorem  C}. In the Appendix, we prove certain density properties and we prove convergence properties of basic hypergeometric series. 

\vspace{0.3cm}

\subsection*{General notation}  $\mathbb{N}$ denotes the natural numbers excluding 0. For weight theory we refer to \cite{TakII}. If $\varphi$ is a normal, semi-finite, faithful weight on a von Neumann algebra $L^\infty(\bG)$, then we denote $J, \nabla, \sigma$ for the modular conjugation, modular operator and modular automorphism group. We use the formal notation $L^2(\bG)$ for the GNS-space and $L^2(\bG) \cap L^\infty(\bG)$ for the elements $x \in L^\infty(\bG)$ for which $\varphi(x^\ast x) < \infty$.  We use $L^1(\bG)$ for the predual of $L^\infty(\bG)$ and $L^1(\bG)^+$ for its positive part. We use the Tomita algebra,   
\[ 
\begin{split}
 \cT_\varphi = \{ x \in L^\infty(\bG) \mid &\: x \textrm{ is analytic for } \sigma \textrm{ and } \\
&\: \sigma_z(x) \textrm{ and } \sigma_z(x)^\ast \textrm{ are in }  L^2(\bG) \cap L^\infty(\bG)   \}.
\end{split}
\]
We freely use Tomita-Takesaki theory, but let us at least recall the following standard facts \cite{TakII}. Let $\varphi$ be a normal, semi-finite, faithful weight with GNS-representation $(L^2(\bG), \pi, \Lambda)$. The domain of $\Lambda$ equals $L^2(\bG) \cap L^\infty(\bG)$ and $\Lambda$ is $\sigma$-weakly/weak (or equivalently $\sigma$-strong-$\ast$/norm) closed. $\cT_\varphi$ is a $\sigma$-weak/norm core for $\Lambda$.  

\vspace{0.3cm}

For $c \in L^\infty(\bG)$ and $\omega \in L^1(\bG)$ we set $(c \cdot \omega)(x) = \omega(xc), x\in L^\infty(\bG)$. Also, for $a,b \in L^2(\bG) \cap L^\infty(\bG)$, we set $(a\varphi b^\ast)(x) = \varphi(b^\ast xa), x \in L^\infty(\bG)$. For $\xi, \eta \in L^2(\bG)$ we set $\omega_{\xi, \eta}(x) = \langle x \xi, \eta \rangle, x \in L^\infty(\bG)$ and $\omega_\xi = \omega_{\xi, \xi}$.

\vspace{0.3cm}

Tensor products are always von Neumann algebraic, unless clearly stated otherwise. 

\section{Locally compact quantum groups}\label{Sect=Lcqg}

For the Kustermans-Vaes definition of a locally compact quantum group, see \cite{KusVae} and \cite{KusVaeII}.   We would recommend \cite{DaeleLcqg} as an introduction, giving an almost self-contained approach to the theory. For a broader introduction we refer to \cite{Tim}.

\vspace{0.3cm}

Throughout the paper $\mathbb{G}$ is a locally compact quantum group. It consists of a von Neumann algebra $L^\infty(\bG)$ and comultiplication $\Delta: L^\infty(\bG) \rightarrow L^\infty(\bG)  \otimes L^\infty(\bG)$ which is implemented by the left multiplicative unitary $W \in B(L^2(\bG) \otimes L^2(\bG))$, 
\[
\Delta(x) = W^\ast (1 \otimes x) W, \qquad x \in L^\infty(\bG). 
\]
Moreover, there exist normal, semi-finite, faithful Haar weights $\varphi$ and $\psi$ on $L^\infty(\bG)$ that satisfy the left and right invariance axioms,
\[
  (\id\otimes \varphi )\Delta(x) = \varphi(x) \: 1_{L^\infty(\bG)}, \qquad  
   (\psi \otimes \id )\Delta(x) = \psi(x) \: 1_{L^\infty(\bG)}, \qquad x \in L^\infty(\bG)^+. 
\]  
We have a GNS-construction $(L^2(\bG), \pi, \Lambda)$  with respect to the left Haar weight. We omit $\pi$ in the notation. The left and right Haar weight are related by the modular element $\delta$,  affiliated with $L^\infty(\bG)$, by the formal identification $\psi(\: \cdot \:) = \varphi(\delta^{\frac{1}{2}} \: \cdot \: \delta^{\frac{1}{2}} )$, see \cite{VaeRad}. The triple  $(L^2(\bG), \id, \Gamma)$ denotes the GNS-construction for $\psi$.  

\vspace{0.3cm}

The quantum group $\bG$ comes with an unbounded antipode $S: ( {\rm Dom}(S) \subseteq L^\infty(\bG) ) \rightarrow L^\infty(\bG)$ that has a unique polar decomposition $S = R \circ \tau_{-i/2}$. Here, $R: L^\infty(\bG) \rightarrow L^\infty(\bG)$ is the unitary antipode and $\tau: \mathbb{R} \rightarrow {\rm Aut}(L^\infty(\bG))$ is the scaling group. We define \cite[Section 4]{KusUniv},
\[
L^1(\bG)^\sharp = \left\{ \omega \in L^1(\bG) \mid \exists \theta \in L^1(\bG) \textrm{ s.t. } (\theta \otimes \id)(W) = (\omega \otimes \id)(W)^\ast \right\}.
\]
In case $\omega \in L^1(\bG)^\sharp$, the corresponding $\theta \in L^1(\bG)$ satisfies $\theta(x) = \overline{\omega( S(x)^\ast )}$ for every $x \in {\rm Dom}(S)$. Conversely, if ${\rm Dom}(S) \rightarrow \mathbb{C}: x \mapsto \overline{\omega( S(x)^\ast )}$ extends boundedly to $L^\infty(\bG)$, then $\omega \in L^1(\bG)^\sharp$ and we denote this extension by $\omega^\ast$.  The scaling constant $\nu\in \mathbb{R}^+$ is then defined by $\varphi \circ \tau_t = \nu^{-t} \varphi$. 

\vspace{0.3cm}

 There exists a Pontrjagin dual quantum group $\hat{\bG}$ and all its associated objects will be equipped with a hat. The left multiplicative unitary  $W \in L^\infty(\bG) \otimes L^\infty(\hat{\bG})$    fully determines $\bG$ as well as $\hat{\bG}$. We have $\hat{W} = \Sigma W^\ast \Sigma$, where $\Sigma: L^2(\bG) \otimes L^2(\bG) \rightarrow  L^2(\bG) \otimes L^2(\bG)$ is the flip on the Hilbert space level.   $L^\infty(\bG)$ is the $\sigma$-strong-$\ast$ closure of, 
\[
 \left\{ (\id \otimes \omega)(W) \mid \omega \in
B(L^2(\bG))_\ast \right\},  
\]
and  $L^\infty(\hat{\bG})$ is the $\sigma$-strong-$\ast$ closure of, 
\[
 \left\{ (\omega \otimes \id) (W) \mid \omega \in
B(L^2(\bG))_\ast \right\}.  
\]
For $\omega \in
L^1(\bG), \theta \in L^1(\hat{\bG})$, we use the standard notation, 
\[
\lambda(\omega) = (\omega
\otimes \id) (W), \qquad \hat{\lambda}(\theta) = (\id \otimes \theta)(W^\ast). 
\]

The dual left Haar weight $\hat{\varphi}$ on $L^\infty(\hat{\bG})$ is constructed as follows. 
We let $\mathcal{I}$ be the set of
$\omega \in L^1(\bG)$, such that $\Lambda(x) \mapsto \omega(x^\ast),
x \in L^2(\bG) \cap L^\infty(\bG)$ extends to a bounded functional on $L^2(\bG)$.   By the
Riesz theorem, for every $\omega \in \mathcal{I}$, there is a
unique vector denoted by $\xi(\omega) \in L^2(\bG)$ such that,
\[
\omega(x^\ast) = \langle \xi(\omega), \Lambda(x)  \rangle, \quad x \in
L^2(\bG) \cap L^\infty(\bG).
\]
The dual left Haar weight $\hat{\varphi}$ is defined to be
the unique normal, semi-finite, faithful weight on $L^\infty(\hat{\bG})$, with
GNS-construction $(L^2(\bG), \iota, \hat{\Lambda})$ such that
$\lambda(\mathcal{I})$ is a $\sigma$-strong-$\ast$/norm core for
$\hat{\Lambda}$ and $\hat{\Lambda}(\lambda(\omega)) = \xi(\omega),
\omega \in \mathcal{I}$. 

We mention that one can also construct a dual right Haar weight, but we do not use it in this paper. 

\vspace{0.3cm}

There is a collection of relations between the objects we introduced so far and they can all be found in \cite{KusVaeII}. We record them here. $P$ is defined by $P^{it} = \nu^{\frac{t}{2}} \Lambda(\tau_t(x)) $.
\[
\begin{array}{lll}
\varphi \circ R =   \psi & \nabla^{is } \delta^{it } =  \nu^{ist} \delta^{it } \nabla^{is } & \tau_t(x) =  \hat{\nabla}^{it} x \hat{\nabla}^{-it} \\
\varphi \circ \tau_t =  \nu^{-t} \varphi & \hat{\nabla}^{is } \delta^{it } = \delta^{it } \hat{\nabla}^{is } &
R(x) =   \hat{J} x^\ast \hat{J}  \\
\psi \circ \tau_t =   \nu^{-t} \psi  & \hat{\nabla}^{it} \nabla^{is} =  \nu^{ist} \nabla^{is} \hat{\nabla}^{it} &  \nabla^{it} = \hat{P}^{it} J \hat{\delta}^{it} J \\
& \hat{J} \delta \hat{J} =   \delta^{-1} & 
\end{array}
\]
Furthermore,
\[
(\tau_t \otimes \hat{\tau}_t)(W) = W, \qquad (R \otimes \hat{R})(W) = W^\ast. 
\]

\vspace{0.3cm}
\hyphenation{al-ge-bras}

Underlying $\bG$ there exist reduced C$^\ast$-algebraic quantum groups of which the C$^\ast$-algebras are defined by,
\[ 
C_0(\bG) =   \textrm{clo} \left\{ (\id
\otimes \omega)(W) \mid \omega \in B(L^2(\bG))_\ast
\right\}, 
\] 
and
\[
C^\ast_r(\bG) = \textrm{ clo} \left\{ (\omega
\otimes \id)(W) \mid \omega \in B(L^2(\bG))_\ast
\right\},  
\]  
where the closures are norm closures in $B(L^2(\bG))$. It is worth mentioning that in this paper,  when working with $C_0(\bG)$, we always add the assumption that it is coamenable (see Lemma \ref{Lem=CoamenabilityEquivalence} for the definition of coamenability). Hence, there is no distinction between reduced and universal in our notation for the C$^\ast$-algebra of $\bG$. 

\vspace{0.3cm}

There is a universal C$^\ast$-algebraic quantum group $C^\ast_u(\bG)$ whose C$^\ast$-algebra is the univeral completion of $L^1(\bG)^\sharp$ equipped with its Banach $\ast$-algebra structure, \cite{KusUniv}. We use $\mathcal{V}$ for Kusterman's universal multiplicative unitary. For coamenable quantum groups it is contained in $\mathcal{M}(C_0(\bG) \otimes C_u^\ast(\bG))$ (with $\mathcal{M}$ the multiplier algebra and $\otimes$ the minimal tensor product). We use the universal objects briefly in the proof of Theorem \ref{Thm=Haagerup}, explaining its universal properties further.

\subsection*{Corepresentations}
An operator $U \in L^\infty(\bG) \otimes B(\cH_U)$ is called a {\it corepresentation} if $(\Delta \otimes \id)(U) = U_{13} U_{23}$. We call a corepresentation $U$ unitary (resp. invertible)    if $U$ is unitary (resp. invertible)  as an operator.   In case $\bG = (L^\infty(G), \Delta_G)$ is a commutative quantum group, then every (bounded) corepresentation is automatically invertible. For arbitrary quantum groups this is more subtle, c.f. \cite{BraDawSam}. Recall that $W$ is a corepresentation that is moreover unitary.

\subsection*{The Fourier algebra and multipliers}  For operator spaces, we refer to \cite{EffRua}. For a locally compact quantum group $\bG$, the predual $L^1(\bG)$ carries a natural operator space structure. Pulling back the comultiplication to $L^1(\bG)$ yields a {\it convolution product} $\Delta_\ast: L^1(\bG) \hat{\otimes} L^1(\bG)\rightarrow L^1(\bG)$ which will usually be denoted by $\ast$. In this way $L^1(\hat{\bG})$ becomes a completely contractive Banach algebra. We will use $A(\bG) =  \hat{\lambda}(L^1(\hat{\bG}))$. And $\Vert \hat{\lambda}(\omega) \Vert_{A(\bG)} := \Vert \omega \Vert_{L^1(\hat{\bG})}$ ($\hat{\lambda}$ is injective). $A(\bG)^+$ is the subset that corresponds to $L^1(\hat{\bG})^+$. 

An operator $b \in L^\infty(\bG)$ is called a {\it left Fourier multiplier} if for every $\omega \in L^1(\hat{\bG})$ there exists a $\theta \in L^1(\hat{\bG})$ such that $\hat{\lambda}(\theta) = b \hat{\lambda}(\omega)$. In this case left multiplication with $b$ defines a bounded map $A(\bG) \rightarrow A(\bG)$ by the closed graph theorem. In case this map is completely bounded, we write $b \in \MCB(\bG)$, i.e. $b$ is a completely bounded left Fourier multiplier. In this paper all multipliers will be left multipliers and therefore we will drop the indication `{\it left}' in our terminology.  For $b \in \MCB(\bG)$, we have, 
\begin{equation}\label{Eqn=MultiplierLinftyNorm}
\Vert b \Vert_{L^\infty(\bG)} \leq \Vert b \Vert_{\MCB(\bG)}.
\end{equation}
It is also useful to remark that $A(\bG) \subseteq \MCB(\bG)$ and for $\omega \in A(\bG)$ we have, 
\[
\Vert \omega \Vert_{A(\bG)} \geq \Vert \omega \Vert_{\MCB(\bG)}.
\]
A completely bounded multiplier $b \in \MCB(\bG)$ is called {\it positive} if it maps $A(\bG)^+$ to $A(\bG)^+$. {\it Complete positivity} is then defined by being positive on every matrix level of the operator space structure.

\section{A sufficient condition for weak amenability of locally compact quantum groups}\label{Sect=CompactaApproximation}

We call a locally compact quantum group $\bG$ {\it weakly amenable} if there exists a $C \in \mathbb{R}$ and a  net $\{b_k\}$ in $A(\bG)$ such that $\Vert b_k \Vert_{\MCB(\bG)} \leq C$  and for every $c \in A(\bG)$, we have,
\begin{equation}\label{Eqn=ApproximationThing}
\Vert b_k c - c \Vert_{A(\bG)} \rightarrow 0.
\end{equation}
The infimum over all $C \in \mathbb{R}$ such that such a net $b_k$ exists is called the {\it Cowling-Haagerup} constant; notation $\Lambda(\bG)$. 
 
\begin{rmk}
The definition of weak amenability first appears in \cite{CanHaa}, where it is proved that $SO_0(n,1)$ is weakly amenable. In Section \ref{Sect=WeakAmenability} we give a more elaborate discussion of examples of weakly amenable (quantum) groups. 
\end{rmk}

In concrete examples \eqref{Eqn=ApproximationThing} can often be hard to check. In this section, we give a sufficient criterium for weak amenability based on \cite{CanHaa}. This is Theorem \ref{Thm=CompactaApproximation}.   
 
Recall that a locally compact quantum group $\bG$ is called {\it coamenable} if there exists a state $\epsilon$ on $C_0(\bG)$ such that $(\epsilon \otimes \id)(W) = 1$,  see also  Lemma \ref{Lem=CoamenabilityEquivalence} for equivalent definitions.  We also need {\it modular} multipliers.

\begin{dfn}
We call a multiplier $b \in \MCB(\bG)$ {\it modular} if for every $t \in \mathbb{R}$ we have   $\tau_t(b) = b$. 
\end{dfn}
\begin{rmk}
Let $b \in \MCB(\bG)$ be a modular multiplier. Consider the mapping $L^1(\hat{\bG}) \rightarrow L^1(\hat{\bG}): \omega \rightarrow  \hat{\lambda}^{-1}(b\hat{\lambda}(\omega))$ and let $\Phi_b: L^\infty(\hat{\bG}) \rightarrow L^\infty(\hat{\bG})$ be its dual. Then, $\Phi_b$ commutes with the modular automorphsim group $\hat{\sigma}$. Indeed, using Lemma \ref{Lem=ModularAppendix} in the second and fourth equation,  
\[
\begin{split}
& \langle \hat{\sigma}_t(\Phi_b(x)), \omega \rangle_{L^\infty(\hat{\bG}), L^1(\hat{\bG})} \\
=& \langle  x, \hat{\lambda}^{-1}(b\hat{\lambda}(\omega \circ \hat{\sigma}_t)) \rangle_{L^\infty(\hat{\bG}), L^1(\hat{\bG})}  \\
= &  \langle  x, \hat{\lambda}^{-1}(b \tau_{-t}(\hat{\lambda}(\omega )) \delta^{it} ) \rangle_{L^\infty(\hat{\bG}), L^1(\hat{\bG})}  \\
= &  \langle  x, \hat{\lambda}^{-1}(\tau_{-t}( b \hat{\lambda}(\omega )) \delta^{it} ) \rangle_{L^\infty(\hat{\bG}), L^1(\hat{\bG})}  \\
= &  \langle  \hat{\sigma}_t( x) ,\hat{\lambda}^{-1}(  b \hat{\lambda}(\omega ))   \rangle_{L^\infty(\hat{\bG}), L^1(\hat{\bG})}  \\
= &  \langle  \Phi_b(\hat{\sigma}_t(x))  , \omega    \rangle_{L^\infty(\hat{\bG}), L^1(\hat{\bG})}.  
\end{split}
\]
This justifies the terminology {\it modular}. 
\end{rmk}

\vspace{0.3cm}

For any $c \in A(\bG)$ we denote $\omega_c \in L^1(\hat{\bG})$ for the functional such that $\hat{\lambda}(\omega_c) = c$.

\begin{lem}\label{Lem=L1Approx}
Let $\bG$ be a coamenable quantum group. Let $b_k \in \MCB(\bG)$ be a net of positive  multipliers, such that for every $c \in C_0(\bG)$ we have $b_k c \rightarrow c$ in the norm of $C_0(\bG)$. Then, for every $c \in A(\bG)^+$ we have $\Vert b_k c \Vert_{A(\bG)} \rightarrow \Vert c \Vert_{A(\bG)}$. 
\end{lem}
\begin{proof}
Take $c \in A(\bG)^+$ so that  $\omega_c$ is a positive functional. In that case also $\omega_{b_k c}$ is positive. Since $\bG$ is coamenable, by definition  there is a bounded positive functional $\epsilon \in C_0(\bG)^\ast$ such that $(\epsilon \otimes \id)(W) = 1$. Then, also $(\epsilon \otimes \id)(W^\ast) = 1$. Then, using the positivity of  $\omega_{b_k c}$ in the first equality and in addition the positivity of $\epsilon$ in the third equality,
\[
\begin{split}
\Vert b_k c \Vert_{A(\bG)} = & \omega_{b_kc}(1) = \langle (\epsilon \otimes \id)(W^\ast), \omega_{b_k c} \rangle_{L^\infty(\hat{\bG}), L^1(\hat{\bG})} \\=& \langle \epsilon, (\id \otimes \omega_{b_k c} )(W^\ast)\rangle_{C_0(\bG)^\ast, C_0(\bG)}= \epsilon(b_k c).
\end{split}
\]
Similarly, $\Vert c \Vert_{A(\bG)} = \epsilon(c)$. Taking the limit $k \rightarrow \infty$, 
\begin{equation} 
\Vert b_k c \Vert_{A(\bG)} 
= \epsilon(b_k c)  \rightarrow \epsilon(c) = 
\Vert c \Vert_{A(\bG)}. 
\end{equation} 
This proves the lemma.
\end{proof}

The following result is the main theorem of this section. It relies on Lemma \ref{Lem=Technical}, which we shall prove in the remainder of this section. Throughout the proof we use convergences in the norm of $L^\infty(\bG)$, which in fact take place in the smalller C$^\ast$-algebra $C_0(\bG)$. We will not incorporate this in the notation. Recall also that the dense subspace, 
\[
L^1(\bG)^\flat \subseteq L^1(\bG),
\]
is defined in Lemma \ref{Lem=TomitaTransform}. Recall also that for $\omega \in L^1(\bG)^\flat$ and $z \in \mathbb{C}$ we denote $\omega_{[z]}\in L^1(\bG)$ for the unique functional such that $\lambda(\omega_{[z]}) = \hat{\sigma}_z( \lambda(\omega))$. In Lemma \ref{Lem=TomitaTransform} we proved that $\mathbb{C} \rightarrow L^1(\bG): z \mapsto \omega_{[z]}$ is analytic for every $\omega \in L^1(\bG)^\flat$.

\begin{thm}\label{Thm=CompactaApproximation}
Let $\mathbb{G}$ be a coamenable quantum group.  Let  $b_k \in A(\bG)$ be a sequence of positive modular multipliers such that $\Vert b_k \Vert_{\MCB(\bG)}$ is bounded and such that for every $c \in C_0(\bG)$ we have $\Vert b_k c - c \Vert_{L^\infty(\bG)} \rightarrow 0$. Then,  for every $c \in A(\bG)$ as $k \rightarrow \infty$, 
\begin{equation}\label{Eqn=Approx}
\Vert b_k c - c \Vert_{A(\bG)} \rightarrow 0.
\end{equation}
That is, $\bG$ is weakly amenable with Cowling-Haagerup constant smaller than or equal to
$\limsup_{k \in K} \Vert b_k \Vert_{\MCB(\bG)}$.
\end{thm}

\begin{rmk}
For clarity, let us make the following remark on the convergence \eqref{Eqn=Approx}. Note that each $b_k$ is in $A(\bG)$ and not in $\MCB(\bG)$. So each $b_k$ corresponds to a functional $\hat{\omega}_k \in L^1(\hat{\bG})^+$ and similarly $c \in A(\bG)$ corresponds to a $\omega_c \in L^1(\hat{\bG})$. The convergence \eqref{Eqn=Approx} can be restated at the $L^1(\hat{\bG})$-level, namely $\Vert \omega_k \ast \omega_c - \omega_c \Vert_{L^1(\hat{\bG})} \rightarrow 0$.  
\end{rmk}
\begin{proof}[Proof of Theorem \ref{Thm=CompactaApproximation}]

%Let $\cJ$ be the  set of all $\omega \in L^1(\bG)^+$ such that $\lambda(\omega) \in \cT_{\hat{\varphi}}^2$. $\hat{\cJ}$ denotes the analogous set for the Pontrjagin dual quantum group. By Lemma \ref{Lem=Density} the linear span of $\hat{\cJ}$ is dense in $L^1(\hat{\bG})$. Hence, to prove \eqref{Eqn=Approx} we may assume that $c = \hat{\lambda}(\omega)$ with $\omega \in \hat{\cJ}$. In particular, we w
Since $\Vert b_k \Vert_{\MCB(\bG)}$ is bounded, it follows from a $3\epsilon$-argument that we may prove \eqref{Eqn=Approx} for a dense set of $c \in A(\bG)^+$.  
  By Lemma \ref{Lem=GapLemma}, in order to prove our theorem, we may assume that $\omega_c(x) = \langle x \hat{\Lambda}(d), \hat{\Lambda}(d) \rangle =:  (d \hat{\varphi} d^\ast)(x)$, with $d = \lambda((\omega^\ast \ast \omega)_{[-i/2]}), \omega \in L^1(\bG)^\flat$. In particular $\hat{\sigma}_{i/2}(d) = \lambda(\omega^\ast \ast \omega) \geq 0$.    
 
 Moreover,  using Lemma \ref{Lem=Technical}, we may assume that there exist $d_k \in \cT_{\hat{\varphi}}$ such that $\Vert \omega_{b_kc} - d_k \hat{\varphi} d_k^\ast \Vert_{L^1(\hat{\bG})} < \frac{1}{k}$ and such that $\hat{\sigma}_{i/2}(d_k)\hat{\sigma}_{i/2}(d_k)^\ast$ is a bounded sequence in $L^\infty(\hat{\bG})$ and that in fact $\hat{\sigma}_{i/2}(d_k) \in L^\infty(\hat{\bG})^+$. We sometimes write  $\hat{\sigma}_{i/2}(d_k)^\ast$ to clarify our equalities, even though this is a positive operator.  

 It suffices  then  to prove that:
\begin{equation}\label{Eqn=AlternativeConvergence}
\Vert d_k \hat{\varphi} d_k^\ast - d \hat{\varphi} d^\ast \Vert_{L^1(\hat{\bG})} \rightarrow 0.
\end{equation}
Since for $x \in L^\infty(\hat{\bG})$ we have that $(d_k \hat{\varphi} d_k^\ast)(x) = \langle x \hat{\Lambda}(d_k), \hat{\Lambda}(d_k) \rangle$ and similarly with $d_k$ replaced by $d$, it suffices to prove that $\hat{\Lambda}(d_k) \rightarrow \hat{\Lambda}(d)$ in the norm of $L^2(\hat{\bG})$. Now,
\[
\begin{split}
& \Vert \hat{\Lambda}(d_k) - \hat{\Lambda}(d) \Vert_{L^2(\hat{\bG})}^2 
= \langle  \hat{\Lambda}(d_k) - \hat{\Lambda}(d),  \hat{\Lambda}(d_k) - \hat{\Lambda}(d)  \rangle
\\ 
= & \langle   \hat{\Lambda}(d_k) ,  \hat{\Lambda}(d_k)  \rangle + \langle   \hat{\Lambda}(d) ,  \hat{\Lambda}(d)  \rangle - 2 \Re \left( \langle \hat{\Lambda}(d_k) ,\hat{\Lambda}(d) \rangle \right). 
\end{split}
\]
Hence, it suffices to prove that,
\begin{equation}\label{Eqn=SufficientCondition}
\Vert \hat{\Lambda}(d_k) \Vert_{L^2(\hat{\bG})} \rightarrow \Vert \hat{\Lambda}(d) \Vert_{L^2(\hat{\bG})} \quad \textrm{ and } \quad \langle \hat{\Lambda}(d_k) - \hat{\Lambda}(d),\hat{\Lambda}(d) \rangle \rightarrow 0.
\end{equation}

For the left condition of \eqref{Eqn=SufficientCondition}, we find 
\[
\begin{split}
& \left| \Vert \hat{\Lambda}(d_k) \Vert_{L^2(\hat{\bG})}^2 -  \Vert \hat{\Lambda}(d) \Vert_{L^2(\hat{\bG})}^2  \right|  
=  \left| (d_k \hat{\varphi} d_k^\ast)(1) - (d \hat{\varphi} d^\ast)(1) \right| \\
\leq & \left| \omega_{b_k c}(1) -  \omega_{c}(1) \right| + \frac{1}{k} 
 =   \left|   \Vert b_k c \Vert_{A(\bG)} - \Vert c \Vert_{A(\bG)} \right| + \frac{1}{k},
\end{split}
\]
and from  from Lemma \ref{Lem=L1Approx} (this is where we use that $\bG$ is coamenable) it follows that this expression converges to 0. Hence, it remains to check the right condition of \eqref{Eqn=SufficientCondition}. 

Let $\theta \in L^1(\bG)^+$.
%, c \in C_0(\bG)$. 
Set $x = \lambda(\theta)$, then as $k \rightarrow \infty$, 
\[
\begin{split}
& \vert \hat{\varphi}(d^\ast x d) - \hat{\varphi}(d_k^\ast x d_k)\vert  
%\leq &  \vert \hat{\varphi}(d^\ast x d) - \omega_{b_k c}( x) \vert + \frac{1}{k} \Vert x \Vert_{L^\infty(\bG)} \\
\leq   \vert \omega_c(x) - \omega_{b_k c}(x) \vert + \frac{1}{k} \Vert x \Vert_{L^\infty(\hat{\bG})}  \\
= &  \vert \theta(c)  - \theta(b_k c) \vert + \frac{1}{k} \Vert x \Vert_{L^\infty(\hat{\bG})}   \rightarrow 0. 
\end{split}
\]   
Since we proved that $d_k \hat{\varphi} d_k^\ast$ is a bounded sequence in $L^1(\hat{\bG})$ (i.e. the left part of \eqref{Eqn=SufficientCondition}) and the span of $\lambda(\theta)$ with $\theta \in L^1(\bG)^+$ is norm dense in $C_r^\ast(\bG)$, we find that in fact, \begin{equation}\label{Eqn=PartialGoToZero}
\vert \hat{\varphi}(d^\ast x d) - \hat{\varphi}(d_k^\ast x d_k)\vert \rightarrow 0 \qquad \textrm{ for every } x \in C^\ast_r(\bG).
\end{equation} 

Let $e,f \in C^\ast_r(\bG)$ be such that $e,f \in \cT_{\hat{\varphi}}$. The existence of such elements is guaranteed by Lemma \ref{Lem=TomitaTransform}. Moreover, such elements are norm dense in $C^\ast_r(\bG)$. Then,
\[
\begin{split}
 & \hat{\varphi}(d_k^\ast e^\ast f d_k)  
=   \hat{\varphi}( \hat{\sigma}_{i/2}(fd_k)  \hat{\sigma}_{-i/2}(d_k^\ast e^\ast)   )  \\
=  &  \langle \hat{\sigma}_{i/2}(d_k)\hat{\sigma}_{i/2}(d_k)^\ast \hat{\Lambda}(\hat{\sigma}_{-i/2}(e^\ast) ), \hat{\Lambda}(\hat{\sigma}_{i/2}(f)^\ast ) \rangle,    
\end{split}
\]  
and similarly,
\[
\begin{split}
 \hat{\varphi}(d^\ast e^\ast f d) = \langle \hat{\sigma}_{i/2}(d)\hat{\sigma}_{i/2}(d)^\ast \hat{\Lambda}(\hat{\sigma}_{-i/2}(e^\ast) ), \hat{\Lambda}(\hat{\sigma}_{i/2}(f)^\ast ) \rangle.
\end{split}
\]
By \eqref{Eqn=PartialGoToZero}, we see that $\vert  \hat{\varphi}(d^\ast e^\ast f d)  -  \hat{\varphi}(d_k^\ast e^\ast f d_k) \vert \rightarrow 0$ and since we assumed that $\hat{\sigma}_{i/2}(d_k)\hat{\sigma}_{i/2}(d_k)^\ast$ is a bounded sequence, we find that for all vectors $\xi, \eta \in L^2(\bG)$, 
\[
 \langle \hat{\sigma}_{i/2}(d_k)\hat{\sigma}_{i/2}(d_k)^\ast \xi, \eta \rangle \rightarrow \langle \hat{\sigma}_{i/2}(d)\hat{\sigma}_{i/2}(d)^\ast \xi, \eta \rangle.
\]
Approximating $t \mapsto \sqrt{t}$ with polynomials on the compact interval from $0$ to the number $\sup_k \Vert\hat{\sigma}_{i/2}(d_k)\hat{\sigma}_{i/2}(d_k)^\ast\Vert_{L^\infty(\bG)}$, this then implies that, 
 \[
 \langle \vert \hat{\sigma}_{i/2}(d_k)^\ast \vert \xi, \eta \rangle \rightarrow \langle \vert \hat{\sigma}_{i/2}(d)^\ast\vert  \xi, \eta \rangle.
\]
Henceforth, using positivity of $\hat{\sigma}_{i/2}(d_k)$ and $\hat{\sigma}_{i/2}(d)$ (see the assumptions in the first paragraph of our proof),
 \[
 \langle  \hat{\sigma}_{i/2}(d_k)^\ast  \xi, \eta \rangle \rightarrow \langle  \hat{\sigma}_{i/2}(d)^\ast  \xi, \eta \rangle.
\]
Now, let again $e, f \in \mathcal{T}_{\hat{\varphi}}$. Then,
\[
\begin{split}
&\langle \hat{\Lambda}(ef), \hat{\Lambda}(d_k) \rangle \\
= & \hat{\varphi}(d_k^\ast ef )  \\
= & \hat{\varphi}(\hat{\sigma}_{i/2}(f) \hat{\sigma}_{-i/2}(d_k^\ast) \hat{\sigma}_{-i/2}(e)     )\\
= & \langle \hat{\sigma}_{i/2}(d_k)^\ast \hat{\Lambda}(\hat{\sigma}_{-i/2}(e)  ), \hat{\Lambda}(\hat{\sigma}_{i/2}(f)^\ast )   \rangle \\
\rightarrow & \langle \hat{\sigma}_{i/2}(d)^\ast \hat{\Lambda}(\hat{\sigma}_{-i/2}(e)  ), \hat{\Lambda}(\hat{\sigma}_{i/2}(f)^\ast )   \rangle \\ 
= & \langle \hat{\Lambda}(ef), \hat{\Lambda}(d) \rangle,
\end{split}
\] 
where the last equation follows by following the first equations backwards but with $d$ instead of $d_k$. Using the proved fact that $\Vert \hat{\Lambda}(d_k)\Vert_{L^2(\bG)}$ is bounded, and the standard fact that $\hat{\Lambda}(\cT_{\hat{\varphi}}^2)$ is dense in $L^2(\hat{\bG})$ we find that, 
\[
\langle \hat{\Lambda}(d_k) - \hat{\Lambda}(d), \xi\rangle \rightarrow 0,
\]
for all vectors $\xi \in L^2(\bG)$. In all, we have proved \eqref{Eqn=SufficientCondition} and hence conclude the theorem. 
\end{proof}

We now prove the technical lemmas that are needed to make the assumptions at the beginning of the proof of Theorem \ref{Thm=CompactaApproximation}. 

\begin{lem}\label{Lem=GapLemma}
The   functionals on $L^\infty(\hat{\bG})$ given by
\[
a \mapsto \langle a \hat{\Lambda}( \lambda(  (\omega^\ast \ast \omega)_{[-i/2]} )), \hat{\Lambda}( \lambda(  (\omega^\ast \ast \omega)_{[-i/2]} )) \rangle,
\]
with $\omega \in L^1(\bG)^\flat$ are dense in $L^1(\hat{\bG})^+$. 
\end{lem}
\begin{proof}
For $x, y \in L^2(\hat{\bG}) \cap L^\infty(\hat{\bG} )$, there exists a unique functional $\hat{\varphi}_{y^\ast x} \in L^1(\hat{\bG})$ such that for $e, f \in \cT_{\hat{\varphi}}$ we have \cite[Proposition 4]{Terp},
\begin{equation}\label{Eqn=NewFunctEqn}
\begin{split}
\hat{\varphi}_{y^\ast x} (e^\ast f) = & \langle \hat{J} x^\ast y \hat{J} \hat{\Lambda}(f), \hat{\Lambda}(e) \rangle \\
= & \langle \hat{J} f^\ast e \hat{J} \hat{\Lambda}(x), \hat{\Lambda}(y) \rangle.
\end{split}
\end{equation}
From the second line of \eqref{Eqn=NewFunctEqn} we infer that $\{ \hat{\varphi}_{x^\ast x} \mid x \in L^2(\hat{\bG}) \cap L^\infty(\hat{\bG}) \}$ is dense in $L^1(\hat{\bG})^+$. (Indeed, every positive normal functional on $L^\infty(\hat{\bG})$ is of the form $\omega_{\xi, \xi} = \omega_{\hat{J} \hat{J} \xi, \hat{J} \hat{J} \xi}$ and $\hat{J} \xi$ can be approximated with $\hat{\Lambda}(x)$ with $x \in L^2(\hat{\bG}) \cap L^\infty(\hat{\bG})$). Since we have $\hat{\varphi}_{x^\ast x} = \hat{\varphi}_{\vert x \vert^2 }$, we see that the functionals given by,
\[
a \mapsto \langle \hat{J} a^\ast \hat{J} \hat{\Lambda}(x), \hat{\Lambda}(x) \rangle,
\] 
with $x \in L^2(\hat{\bG}) \cap L^\infty(\hat{\bG})^+$ are dense in $L^1(\hat{\bG})^+$. 

Now, let $x \in  L^2(\hat{\bG}) \cap L^\infty(\hat{\bG})^+$ with spectral decomposition $x = \int_{0}^\infty \lambda dE_x(\lambda)$. Put $x_n = \int_{1/n}^\infty \lambda dE_x(\lambda)$. Then $\Vert x - x_n \Vert_{L^\infty(\hat{\bG})} \rightarrow 0$ and clearly $x_n \leq x$. Moreover, this implies that $x_n \in  L^2(\hat{\bG}) \cap L^\infty(\hat{\bG})^+$, and 
\[
\Vert \hat{\Lambda}(x_n) - \hat{\Lambda}(x) \Vert^2_{L^2(\hat{\bG})} = \hat{\varphi}( x^2) + \hat{\varphi}(x_n^2) - \hat{\varphi}(x x_n) - \hat{\varphi}(x_n x) \rightarrow 0, 
\]
since $\hat{\varphi}$ is normal and $\sup_n (x_n^2) = x^2 = \sup_n (x_n x)$. The conclusion is that the  functionals 
\[
a \mapsto \langle \hat{J} a^\ast \hat{J} \hat{\Lambda}(x), \hat{\Lambda}(x) \rangle,
\]  
with 
\begin{equation}\label{Eqn=XCondition}
x \in L^2(\hat{\bG}) \cap L^\infty(\hat{\bG})^+ \quad \textrm{ such that } x = \int_{1/n}^\infty \lambda dE_x(\lambda) \quad  \textrm{ for some } \quad n \in \mathbb{N},
\end{equation}
are dense in $L^1(\hat{\bG})^+$. 

Now, let $x$ indeed satisfy the condition \eqref{Eqn=XCondition}. Put $y = \sqrt{x}$. Then, $y \leq  \sqrt{n} x$, so that $y \in L^2(\hat{\bG}) \cap L^\infty(\hat{\bG})^+$. Since $\lambda( L^1(\bG)^\flat ) $ forms a $\sigma$-strong-$\ast$/norm core for $\hat{\Lambda}$, c.f. Lemma \ref{Lem=TomitaTransform}, we can find a net $\{ \omega_i \}$ in $L^1(\bG)^\flat$, such that $\lambda(\omega_i)$ is a (bounded) net converging to $y$ in the $\sigma$-strong-$\ast$ topology and $\xi(\omega_i) = \hat{\Lambda}(\lambda(\omega_i)) \rightarrow \hat{\Lambda}(y)$ in norm. Then, a $2\epsilon$-estimate shows that $\hat{\Lambda}( \lambda( \omega_i^\ast \ast \omega_i) ) = \lambda(\omega_i)^\ast \hat{\Lambda}(\lambda(\omega_i)) \rightarrow y^\ast \hat{\Lambda}(y) = \hat{\Lambda}(x)$ in norm. From this it follows that the functionals,
\[
a \mapsto \langle \hat{J} a^\ast \hat{J} \hat{\Lambda}( \lambda( \omega^\ast \ast \omega ) ), \hat{\Lambda}(\lambda(\omega^\ast \ast \omega)) \rangle, 
\]   
with $\omega \in L^1(\bG)^\flat$ are dense in $L^1(\hat{\bG})^+$. But, using that $\hat{J} = \hat{\nabla}^{1/2} \hat{J} \hat{\nabla}^{1/2}$ and $\lambda(\omega^\ast \ast \omega) \geq 0$,  
\[
\begin{split}
& \langle \hat{J} a^\ast \hat{J} \hat{\Lambda}( \lambda( \omega^\ast \ast \omega ) ), \hat{\Lambda}(\lambda(\omega^\ast \ast \omega)) \rangle \\
= &  \langle  a \hat{J} \hat{\Lambda}( \lambda( \omega^\ast \ast \omega ) ), \hat{J} \hat{\Lambda}(\lambda(\omega^\ast \ast \omega)) \rangle \\
= &  \langle  a \hat{\nabla}^{1/2} \hat{\Lambda}( \lambda( \omega^\ast \ast \omega ) ), \hat{\nabla}^{1/2} \hat{\Lambda}(\lambda(\omega^\ast \ast \omega)) \rangle \\
= &  \langle  a   \hat{\Lambda}( \lambda( (\omega^\ast \ast \omega)_{[-i/2]} ) ), \hat{\Lambda}(\lambda( (\omega^\ast \ast \omega)_{[-i/2]} )) \rangle, 
\end{split}
\]
which concludes our lemma.
\end{proof}

\begin{lem}\label{Lem=Modular}
Let $\omega \in L^1(\bG)^\flat$ and let $b \in \MCB(\bG)$ be a modular multiplier. We have $b \cdot \omega  \in \cI$ and $\lambda(b \cdot \omega )$ is analytic for $\hat{\sigma}$.  For every $z \in \mathbb{C}$, we have $\hat{\sigma}_z(\lambda(b\cdot \omega)) = \lambda(b \cdot \omega_{[z]})$.
\end{lem}
\begin{proof}
By Lemma \ref{Lem=ModularAppendix} we find that we must have   $\theta_{[t]}(x) = \theta(\delta^{it} \tau_{-t}( x ))$ for every $t \in \mathbb{R}$ and $\theta \in L^1(\bG)^\flat$. Applying this to $\theta = b \cdot \omega$, this means that 
\[
\begin{split}
& \langle x, \theta_{[t]} \rangle = \langle  \delta^{it} \tau_{-t}( x ), \theta \rangle = 
 \langle  \delta^{it} \tau_{-t}( x ), b \cdot \omega \rangle =
 \langle  \delta^{it} \tau_{-t}( x ) b,  \omega \rangle \\ = &
\langle  \delta^{it} \tau_{-t}( x b),  \omega \rangle =
\langle xb, \omega_{[t]} \rangle = 
\langle x, b \cdot \omega_{[t]} \rangle. 
\end{split}
\]
So that 
\begin{equation}\label{Eqn=AnalEqn}
\hat{\sigma}_t(\lambda(b\cdot \omega)) = \hat{\sigma}_t(\lambda(\theta)) = \lambda(\theta_{[t]}) = \lambda(b \cdot \omega_{[t]}).
\end{equation}  
%\begin{equation}\label{Eqn=AnalEqn}
%\hat{\sigma}_t(\lambda(b\cdot \omega)) = \lambda((b\cdot \omega \cdot \delta^{it}) \circ \tau_{-t}) =  \lambda(b\cdot ((\omega \cdot \delta^{it}) \circ \tau_{-t})) = \lambda(b \cdot \omega_{[t]}).
%\end{equation}
Since, $t \mapsto \omega_{[t]}$ extends analytically to $\mathbb{C} \rightarrow L^1(\bG)$ it follows that  \eqref{Eqn=AnalEqn} is analytic on $\mathbb{C}$. This proves the lemma, since $\mathcal{I}$ is a left $L^\infty(\bG)$-module \cite[Result 8.6]{KusVae}.  
\end{proof}

\begin{lem}\label{Lem=L1CohenTypeLemma}
Let $b_k \in \MCB(\bG)$ be a bounded sequence of multipliers such that for every $c \in C_0(\bG)$ we have $\Vert b_k c - c \Vert_{L^\infty(\bG)} \rightarrow 0$. For every $\omega \in L^1(\bG)$, we have $\Vert b_k \cdot \omega - \omega \Vert_{L^1(\bG)} \rightarrow 0$. 
\end{lem}
\begin{proof}
Let $\xi, \eta \in L^2(\bG)$ and let $c \in C_0(\bG)$. Then, 
\[
\begin{split}
\Vert b_k \cdot \omega_{  c \xi, \eta} - \omega_{ c \xi, \eta} \Vert_{L^1(\bG)} = & \Vert   \omega_{b_k c \xi, \eta} - \omega_{c \xi,  \eta} \Vert_{L^1(\bG)}\\
 \leq & \Vert \xi \Vert_{L^2(\bG)} \Vert \eta \Vert_{L^2(\bG)} \Vert b_k c - c \Vert_{L^\infty(\bG)}  \rightarrow 0.
\end{split}
\]
Let $\xi, \eta \in L^2(\bG)$ again be arbitrary. By Cohen factorization, there exists a  $c \in C_0(\bG)$ and $\xi' \in L^2(\bG)$ such that $\xi = c \xi'$ with $\Vert c \Vert \leq 1$. The lemma then follows from the previous computation. 

\begin{comment}Since $C_0(\bG)$ acts non-degenerately on $L^2(\bG)$, there exists a $c \in C_0(\bG)$ and $\xi' \in L^2(\bG)$ such that $\Vert \xi - c \xi' \Vert_{L^2(\bG)} < \epsilon$. Then, using \eqref{Eqn=MultiplierLinftyNorm},
\[
\begin{split}
&\Vert b_k \cdot \omega_{\xi,  \eta} - \omega_{\xi,  \eta} \Vert_{L^1(\bG)} \\
\leq &\Vert b_k \cdot \omega_{\xi,   \eta} - b_k \cdot \omega_{c \xi', c \eta}  \Vert_{L^1(\bG)}
+ \Vert b_k \cdot \omega_{c \xi',   \eta} - \omega_{c \xi',  \eta}  \Vert_{L^1(\bG)}
+ \Vert  \omega_{c \xi',   \eta} - \omega_{\xi, \eta}  \Vert_{L^1(\bG)} \\
\leq &  \epsilon (1+\Vert b_k \Vert_{L^\infty(\bG)}) \Vert \xi \Vert_{L^2(\bG)} \Vert \eta \Vert_{L^2(\bG)}  +  \Vert b_k \cdot \omega_{c \xi',    \eta} - \omega_{c \xi',  \eta}  \Vert_{L^1(\bG)},
\end{split}
\] 
which becomes arbitrarily small. 
\end{comment}
\end{proof}

Let $x,y \in L^2(\bG) \cap L^\infty(\bG)$. It is proved in \cite[Proposition 4]{Terp} that there exists a unique functional $\varphi_{y^\ast x } \in L^1(\bG)$ such that  for every $e,f \in \cT_{\hat{\varphi}}$ we have,
\begin{equation}\label{Eqn=MagicFunctional}
\varphi_{y^\ast x }(e^\ast f) = \langle J x^\ast y J \Lambda(f), \Lambda(e) \rangle = \langle J f^\ast e J \Lambda(x), \Lambda(y) \rangle.
\end{equation}
We now prove the following technical lemma. It plays a crucial role in Theorem \ref{Thm=CompactaApproximation}. Remark \ref{Rmk=TechnicalNature} explains its nature. 

\begin{lem}\label{Lem=Technical}
Let $b_k \in \MCB(\bG)$ be a bounded sequence of positive modular multipliers such that for every $c \in C_0(\bG)$ we have $\Vert b_k c - c \Vert_{L^\infty(\bG)} \rightarrow 0$. Let  $\omega \in L^1(\bG)^{\flat}$ and set $d = \lambda(\omega)$. In particular, $d \in \cT_{\hat{\varphi}}$ and $d\hat{\varphi}d^\ast \in L^1(\hat{\bG})$ is well defined.  Let $\hat{\theta}_k \in L^1(\hat{\bG})$ be such that $\hat{\lambda}(\hat{\theta}_k) = b_k \hat{\lambda}(d \hat{\varphi} d^\ast)$. Then, for every $k$, there exists an operator $d_k \in \cT_{\hat{\varphi}}$ such that 
\begin{equation}\label{Eqn=TechnicalConvergence}
\Vert d_k \hat{\varphi} d_k^\ast - \hat{\theta}_k \Vert_{L^1(\hat{\bG})} < \frac{1}{k},
\end{equation}
and moreover, $\hat{\sigma}_{i/2}(d_k) \hat{\sigma}_{i/2}(d_k)^\ast$ is bounded in $L^\infty(\bG)$ and  $\hat{\sigma}_{i/2}(d_k) \in L^\infty(\bG)^+$. 
\end{lem}
\begin{proof}
Note that the key properties we need are boundedness of $\hat{\sigma}_{i/2}(d_k) \hat{\sigma}_{i/2}(d_k)^\ast$  and  positivity of $\hat{\sigma}_{i/2}(d_k) \in L^\infty(\bG)^+$. 
The following is the main point of our set-up.

\vspace{0.3cm}

\noindent {\bf Claim 1:} For every $k \in \mathbb{N}$, there exists a unique element $x_k \in L^\infty(\hat{\bG})$ such that for every $e,f \in \cT_{\hat{\varphi}}$ we have,
\begin{equation}\label{Eqn=XkFormula}
\hat{\theta}_k(e^\ast f) = \langle \hat{J} x_k^\ast \hat{J} \hat{\Lambda}(f), \hat{\Lambda}(e) \rangle.
\end{equation}
Moreover,   $x_k$ is positive and the sequence is bounded in $L^\infty(\hat{\bG})$. 

\vspace{0.3cm}

\noindent {\bf Proof of   claim 1:} If such $x_k$ exists, then it is unique by \eqref{Eqn=XkFormula}. Since $b_k$ is a positive multiplier, $\hat{\theta}_k$ is a positive functional. Then, $x_k$ is  positive, since for $e \in \cT_{\hat{\varphi}}$, 
\[
\langle  x_k  \hat{J} \hat{\Lambda}(e), \hat{J} \hat{\Lambda}(e) \rangle =\langle   \hat{J} x_k^\ast  \hat{J} \hat{\Lambda}(e),\hat{\Lambda}(e) \rangle= \hat{\theta}_k(e^\ast e) \geq 0. 
\]  
We use the notation $\omega' = (\omega^\ast)_{[-i]}$. We claim that we may take:
\[
x_k :=  
   \lambda\left( b_k \cdot\left( \left( \omega \ast \omega' \right)_{[i/2]}\right)\right).
\]
 It follows from Lemma \ref{Lem=L1CohenTypeLemma} that $b_k \cdot\left( \left( \omega \ast \omega' \right)_{[i/2]}\right)$ is a bounded sequence in $L^1(\bG)$ and hence $x_k$ is a bounded sequence in $L^\infty(\hat{\bG})$. In fact, $x_k \in L^2(\hat{\bG}) \cap L^\infty(\hat{\bG}) $ by Lemma \ref{Lem=Modular}. We claim that,
\begin{equation}\label{Eqn=Dualizing}
\xi( \omega \ast \omega') = \hat{\Lambda}( \lambda(\omega) \lambda(\omega') ) = \hat{\xi}( \lambda(\omega)  \hat{\varphi} \lambda(\omega) ^\ast) = \hat{\xi}( d \hat{\varphi} d^\ast).
\end{equation}
Indeed, the first and last equation follow by definition. The second equation follows from the fact that for $e\in \cT_{\hat{\varphi}}$ we have, 
\begin{equation}\label{Eqn=L2Check}
\begin{split}
&(\lambda(\omega)  \hat{\varphi} \lambda(\omega) ^\ast)(e^\ast) = \hat{\varphi}(\lambda(\omega) ^\ast e^\ast \lambda(\omega)) = \hat{\varphi}(e^\ast \lambda(\omega)\hat{\sigma}_{-i}(\lambda(\omega) ^\ast ))\\
 =& \hat{\varphi}(e^\ast \lambda(\omega)\lambda(\omega')) = \langle \hat{\Lambda}(\lambda(\omega)\lambda(\omega')), \hat{\Lambda}(e) \rangle.
\end{split}
\end{equation}
Using the fact that $\cT_{\hat{\varphi}}$ is a $\sigma$-weak/norm core for $\hat{\Lambda}$, this yields that \eqref{Eqn=L2Check} holds for all $e \in L^2(\hat{\bG}) \cap L^\infty(\hat{\bG})$, and hence \eqref{Eqn=Dualizing} follows. 

We now prove \eqref{Eqn=XkFormula}. Let $e, f \in \cT_{\hat{\varphi}}$. Then, explaining the equations below, we have
\[
\begin{split}
& \langle \hat{J} x_k^\ast \hat{J} \hat{\Lambda}(f), \hat{\Lambda}(e) \rangle \\
= & \langle x_k \hat{\Lambda}(\hat{\sigma}_{i/2}(e)^\ast ), \hat{\Lambda}( \hat{\sigma}_{i/2} (f)^\ast )\\
= & \hat{\varphi}(\hat{\sigma}_{i/2}(f) x_k \hat{\sigma}_{i/2}(e)^\ast ) \\
= & \hat{\varphi}(\hat{\sigma}_{i/2}(e^\ast f) x_k   )  \\
= & \langle \hat{\Lambda}(x_k), \hat{\Lambda}( \hat{\sigma}_{-i/2}(f^\ast e)  ) \rangle\\
= & \langle b_k \xi(\left( \omega \ast \omega' \right)_{[i/2]} ), \hat{\Lambda}(\hat{\sigma}_{-i/2}(f^\ast e))\rangle\\
= & \langle  b_k \hat{\nabla}^{-\frac{1}{2}} \xi( \omega \ast \omega' ),\hat{\nabla}^{\frac{1}{2}} \hat{\Lambda}(f^\ast e)\rangle\\
= & \langle  b_k \xi( \omega \ast \omega' ), \hat{\Lambda}(f^\ast e)\rangle\\
= & \langle b_k \hat{\xi}(d \hat{\varphi} d^\ast), \hat{\Lambda}(f^\ast e) \rangle \\
= & \langle \Lambda(b_k \hat{\lambda}(d\hat{\varphi} d^\ast) ), \hat{\Lambda}(f^\ast e) \rangle \\
= & \langle \Lambda(\hat{\lambda}(\hat{\theta}_k) ), \hat{\Lambda}(f^\ast e) \rangle \\
= & \langle \hat{\xi}(\theta_k), \hat{\Lambda}(f^\ast e) \rangle \\
= & \hat{\theta}_k(e^\ast f).
\end{split}
\]
The first four equations follow from Tomita-Takesaki theory; the fifth equation is the definition of $x_k$ and the  fact that $\xi(b_k\cdot \theta) = b_k \xi(\theta)$ for every $\theta \in \cI$ (\cite[Result 8.6]{KusVae}); the sixth equation is Tomita-Takesaki theory; the seventh equation is Tomita-Takesaki theory and the fact that since $b_k$ is modular, we have $\tau_t(b_k) = \hat{\nabla}^{it} b_k \hat{\nabla}^{-it} = b_k$ for $t \in \mathbb{R}$; the eight equation is \eqref{Eqn=Dualizing}; the remaing equations follow by the definitions of their objects.  Note that in the eleventh equation we have used Lemma \ref{Lem=L2Thingy}.
This proves the claim.

\vspace{0.3cm}

\noindent {\bf Claim 2:} There exists  a  sequence $y_k \in L^2(\hat{\bG}) \cap L^\infty(\hat{\bG})$ that is bounded with respect to the norm of $L^\infty(\hat{\bG})$ and such that $\Vert \hat{\varphi}_{y_{k}^\ast y_{k}} - \hat{\theta}_k \Vert_{L^1(\hat{\bG})} \rightarrow  0$. See \eqref{Eqn=MagicFunctional} for the definition of $\hat{\varphi}_{y_{k}^\ast y_{k}}$.  Moreover, we may take $y_k \in L^\infty(\hat{\bG})^+$. 

\vspace{0.3cm}

\noindent {\bf Proof of   claim 2:} Let $\{e_j\}$ be a net in $\cT_{\hat{\varphi}}$ such that $e_j \rightarrow 1$ in the $\sigma$-weakly topology and such that $\Vert \hat{\sigma}_z(e_j) \Vert_{L^\infty(\hat{\bG})} \leq e^{\Im(z)^2}$, see \cite[Lemma 9]{Terp}. For every $k$, we let $y_{k,j} \in L^2(\hat{\bG}) \cap L^\infty(\hat{\bG})$ be such that,
\begin{equation}\label{Eqn=PositiveEqn}
y_{k,j}^\ast y_{k,j} = \hat{\sigma}_{i/2}(e_j)  x_k   \hat{\sigma}_{i/2}(e_j)^\ast.
\end{equation}  
The proof of \cite[Theorem 8, p. 331-332]{Terp} yields that,
\[
\begin{split}
&\Vert y_{k,j}^\ast y_{k,j}   \Vert_{L^\infty(\hat{\bG})} <   e^{\frac{1}{2}} \Vert x_k \Vert_{L^\infty(\hat{\bG})}, \\
&\lim_{j \in J} \Vert \hat{\varphi}_{y_{k,j}^\ast y_{k,j}} - \hat{\theta}_k \Vert_{L^1(\hat{\bG})}= 0.
\end{split}
\]
For every $k$, we choose $j_k \in J$ such that $\Vert \hat{\varphi}_{y_{k,j_k}^\ast y_{k,j_k}} - \hat{\theta}_k \Vert_{L^1(\hat{\bG})} < \frac{1}{k}$. We then set  $y_k := y_{k, j_k}$. Note that we could have choosen $y_k$ to be positive in \eqref{Eqn=PositiveEqn} since $x_k$ was positive.

\vspace{0.3cm}

\noindent {\bf Claim 3:} There exists  a  sequence $z_k \in \cT_{\hat{\varphi}}$ that is bounded with respect to the norm of $L^\infty(\hat{\bG})$ and such that $\Vert \hat{\varphi}_{z_{k}^\ast z_{k}} - \hat{\theta}_k \Vert_{L^1(\hat{\bG})} \rightarrow  0$.  Morever, we may take $z_k \in L^\infty(\hat{\bG})^+$. 

\vspace{0.3cm}

\noindent {\bf Proof of   claim 3:} Let $y_k$ be as in {\bf Claim 2}. We assume that $y_k \in L^\infty(\hat{\bG})^+$.
\begin{comment} Let $\{e_j\}$ be a net in $\cT_{\hat{\varphi}}$ such that $e_j \rightarrow 1$ and $ \hat{\sigma}_{i/2}( e_j)^\ast \rightarrow 1$ in the strong  topology and such that $\hat{\sigma}_z(e_j) \leq e^{\Im(z)^2}$, see \cite[Lemma 9]{Terp}. Put $y_{k,j} = e_j y_k e_j^\ast$. Clearly, $y_{k,j}$ is bounded in $L^\infty(\hat{\bG})$. If $y_k$ is positive, then so is $y_{k,j}$. Moreover, using \cite[Proposition 4]{Terp} we see,
\[
\begin{split}
\hat{\varphi}_{y_k^\ast  y_k}(e^\ast f) = & \langle \hat{J}  f^\ast e \hat{J}  \hat{\Lambda}(y_k),  \hat{\Lambda}(y_k) \rangle, \\
\hat{\varphi}_{ e_j y_k^\ast e_j^\ast e_j y_k e_j^\ast}(e^\ast f) = & \langle \hat{J}  f^\ast e \hat{J} e_j\hat{\Lambda}(y_k e_j^\ast), e_j \hat{\Lambda}(y_k e_j^\ast) \rangle \\
= &   \langle \hat{J}  f^\ast e \hat{J} e_j   \hat{J} \hat{\sigma}_{i/2}( e_j)^\ast \hat{J} \hat{\Lambda}(y_k ), e_j \hat{J} \hat{\sigma}_{i/2}( e_j)^\ast \hat{J}  \hat{\Lambda}(y_k  ) \rangle, 
\end{split}
\]
Since $e_j \rightarrow 1$ strongly, this implies that $\Vert \hat{\varphi}_{y_k^\ast  y_k} - \hat{\varphi}_{y_k^\ast e_j^\ast e_j y_k}\Vert_{L^1(\hat{\bG})} \rightarrow 0$. The conclusion is that we may assume in {\bf Claim 2} that both  $y_k$ and $y_k^\ast$  lie in   $L^2(\hat{\bG}) \cap L^\infty(\hat{\bG})$.
\end{comment}
We then define, using a $\sigma$-weak integral,
\[
z_{k,n} = \sqrt{\frac{n}{\pi}} \int_{-\infty} ^\infty e^{-nt^2} \hat{\sigma}_t(y_k) dt\quad \in L^\infty(\bG)^+,
\]
and by standard arguments (see for instance the proof of \cite[Lemma 9]{Terp}) we see that $z_{k,n} \in \cT_{\hat{\varphi}}$.   Moreover,
\[
\hat{\Lambda}(z_{k,n}) =  \sqrt{\frac{n}{\pi}} \int_{-\infty} ^\infty e^{-nt^2} \hat{\nabla}^{it} \hat{\Lambda}(y_k) dt,
\] 
where the integral is a Bochner integral in $L^2(\hat{\bG})$ and as $n \rightarrow \infty$ we find that $\hat{\Lambda}(z_{k,n}) \rightarrow \hat{\Lambda}(y_k)$ in  norm. Recalling that,
\[
\begin{split}
\hat{\varphi}_{z_{k,n}^\ast z_{k,n}}(e^\ast f) = & \langle \hat{J}  f^\ast e \hat{J}  \hat{\Lambda}(z_{k,n}),  \hat{\Lambda}(z_k) \rangle, \\
\hat{\varphi}_{y_{k}^\ast y_{k}}(e^\ast f) = & \langle \hat{J}  f^\ast e \hat{J}  \hat{\Lambda}(y_{k}),  \hat{\Lambda}(y_k) \rangle, \\
\end{split}
\]
this implies that taking the limit $n \rightarrow \infty$, we get $\Vert \hat{\varphi}_{z_{k,n}^\ast z_{k,n}} - \hat{\varphi}_{y_{k}^\ast y_{k}} \Vert_{L^1(\hat{\bG})} \rightarrow 0$. {\bf Claim 3} then follows from {\bf Claim 2}.

\vspace{0.3cm}

\noindent {\bf Proof of Lemma \ref{Lem=Technical}.} Let $z_k$ be as in {\bf Claim 3}. We assume that $z_k \in L^\infty(\hat{\bG})^+$. We set $d_k = \hat{\sigma}_{-i/2}(z_k)$. It follows from {\bf Claim 3} that $d_k$ satisfies the required boundedness and positivity properties in the statement of our lemma. That is, we require that $\hat{\sigma}_{i/2}(d_k) = z_k$ is bounded and positive, which follows from {\bf Claim 3}. Finally, for $e, f \in \cT_{\hat{\varphi}}$ we find,
\[
\begin{split}
(d_k \hat{\varphi}  d_k^\ast)(e^\ast f) = & \hat{\varphi}( d_k^\ast e^\ast f d_k  ) \\
= & \langle \hat{\Lambda}(fd_k), \hat{\Lambda}(ed_k) \rangle \\
= &  \langle \hat{J} \hat{\nabla}^{\frac{1}{2}}  d_k^\ast \hat{J} \hat{\nabla}^{\frac{1}{2}} \hat{\Lambda}(f), \hat{J} \hat{\nabla}^{\frac{1}{2}}  d_k^\ast \hat{J} \hat{\nabla}^{\frac{1}{2}}  \hat{\Lambda}(e) \rangle  \\ 
= & \langle \hat{J} \hat{\sigma}_{i/2}(d_k)  \hat{\sigma}_{i/2}(d_k)^\ast \hat{J}  \hat{\Lambda}(f), \hat{\Lambda}(e) \rangle \\
= &  \langle \hat{J} z_k z_k^\ast \hat{J}  \hat{\Lambda}(f), \hat{\Lambda}(e) \rangle
\end{split}
\]
Hence, we find that $d_k \hat{\varphi} d_k^\ast= \hat{\varphi}_{z_{k} z_{k}^\ast} = \hat{\varphi}_{z_{k}^\ast z_{k}}$. Then, \eqref{Eqn=TechnicalConvergence} follows from {\bf Claim 3}. 
\end{proof}
\begin{rmk}
Theorem \ref{Thm=CompactaApproximation} and Lemma \ref{Lem=Technical} are true if sequences are replaced by nets as it directly follows from the proofs. 
\end{rmk}
\begin{rmk}\label{Rmk=TechnicalNature}
Since the statement and the proof of Lemma \ref{Lem=Technical} are technical in nature, it is useful to comment on its origin.

Suppose that $\bG$ comes from an abelian group $G$. For simplicity, suppose that $\omega \in L^1(G)$ is a compactly supported function. Let $d = \mathcal{F}(\omega)$ be its Fourier transform. In this case $d \hat{\varphi} d^\ast$ corresponds to a function $f \in L^1(\hat{G})$. 

The proof of Lemma \ref{Lem=Technical}, Claim 1 proceeds as follows. Take the Fourier transform $\hat{\mathcal{F}}(f) \in L^\infty(G)$ of $f$. One shows that in fact, $\hat{\mathcal{F}}(f) \in L^1(G)$. Then we multiply $\hat{\mathcal{F}}(f)$  with the multiplier $b_k \in L^\infty(G)$. Next, we take the dual Fourier transform, resulting in the function $(\mathcal{F} \circ b_k \circ \hat{\mathcal{F}})(f) \in L^\infty(\hat{G})$. Because of our choices, $b_k \hat{\mathcal{F}}(f)$ turns out to be a bounded sequence of functions in $L^1(G)$ and henceforth $(\mathcal{F} \circ b_k \circ \hat{\mathcal{F}})(f)\: (= x_k \textrm{ of Claim 1})$ is a bounded sequence in $L^\infty(\hat{G})$. Claims 2 and 3 are then standard approximation methods. 

Claim 1 of the proof of Lemma \ref{Lem=Technical} is proved in exactly the way described above. Note that in the process we used that $\hat{\mathcal{F}}(f) \in L^\infty(G)$ is in fact in $L^1(G)$. For quantum groups the intersection of $L^\infty(\bG)$ and $L^1(\bG)$ has a proper interpretation in terms of compatible couples of non-commutative $L^p$-spaces, see \cite[Section 3]{CasLpf}. We use these ideas implicitly while passing from $L^\infty(\bG)$ to $L^1(\bG)$
\end{rmk}

\begin{rmk}
Let us make the following heuristic comment on why we need the modular condition on the multipliers $b_k$. 
We use the language of compatible couples of non-commutative $L^p$-space for which we refer to \cite{Kos}.

 Along the proof of Lemma \ref{Lem=Technical}, we use a transition between the left injection and the symmetric injection of non-commutative $L^p$-spaces. It was shown in \cite[Theorem 7.1]{CasLpf} that Fourier transforms only exist for the left injection. However, in the proof of Theorem \ref{Thm=CompactaApproximation} we need to work with an injection that has the property that an $x \in L^1(\bG) \cap L^\infty(\bG)$ is positive in $L^1(\bG)$ if and only if it is positive in $L^\infty(\bG)$. The left injection does not have this property, but the symmetric injection does. The transition between the left and symmetric injection causes that we need the modular condition on our multipliers $b_k$. 

We do not know if the modular assumption on $b_k$ is strictly necessary. However, in our example this condition is easy to check. We expect that in similar examples for which Theorem \ref{Thm=CompactaApproximation} is applicable this will not be different.  
\end{rmk}

\section{Basic hypergeometric series and $\SUone$}\label{Sect=SUone}

This section recalls the preliminaries on basic hypergeometric series   \cite{GasRah} and the definition of the extended quantum $SU(1,1)$ group \cite{KoeKus}, \cite{GroKoeKus}. Though that some of the definitions below can be extended for other $q$, we always assume that $0 < q < 1$.  For $k \in \mathbb{N} \cup \{0,\infty\}$, we set the $q$-factorial,
\[
(a; q)_k = \prod_{l=0}^{k-1} (1-aq^l), \qquad (a_0, \ldots, a_n; q)_k = (a_0;q)_k \cdot \ldots \cdot (a_n;q)_k.
\]
We need the following {\it $\theta$-product identity}. 
%\begin{prop}\label{PropThetaPrd}
For $a \in \mathbb{C} \backslash \{ 0 \}, k \in \mathbb{Z}$, 
\begin{equation}\label{EqnIntThePro}
(aq^k, q^{1-k}/a ; q)_\infty = (-a)^{-k} q^{-\frac{1}{2}k(k-1)} (a, q/a;q)_\infty,
\end{equation}
%\end{prop}
Recall that the basic hypergeometric $_2 \! \varphi_1$-function is expressed by:
\[
_2 \! \varphi_1\left( \begin{array}{c} a, b \\ c \end{array};q, z \right) = \prod_{k=0}^\infty \frac{(a,b;q)_k}{(c, q;q)_k} z^{k}.
\] 
%We also need the $C$-function from \cite{GroKoeKus}.
%\[
%C(\mu(\lambda)) = \ep^{\frac{1}{2}(1-\sgn(p_1))}   \eta^{\frac{1}{2}(1-\sgn(p_2)) +n } S(-\sgn(p_1p_2)\lambda; p_1, p_2, n) \frac{A(\lambda; p, m', \ep', \eta') }{A(\lambda; p, m, \ep, \eta)}
%\]
We use the following notation, following \cite{GroKoeKus}:

\vspace{0.3cm}

\begin{tabular}{ll}
  $\mu: \mathbb{C} \backslash \{ 0 \} \rightarrow \mathbb{C} \backslash \{ 0 \}: y \mapsto \frac{1}{2} (y + y^{-1})$ ,    &
   $\chi: -q^{\mathbb{Z}} \cup q^{\mathbb{Z}}: p \mapsto ^q\!\!\log(\vert p\vert)$, 	 \\
   $\nu: -q^{\mathbb{Z}} \cup q^{\mathbb{Z}} \rightarrow \mathbb{R}: t \mapsto q^{\frac{1}{2}(\chi(t) - 1)(\chi(t)-2) }$, & $\kappa: \mathbb{R} \rightarrow \mathbb{R}: x \mapsto \sgn(x) x^2$, \\
$c_q = (\sqrt{2} q (q^2, -q^2; q^2)_\infty)^{-1}$.
\end{tabular} 

\vspace{0.3cm}

\noindent Furthermore, we set $I_q = q^{\mathbb{Z}} \cup -q^{\mathbb{N}}$ and anticipating to the definition of $\SUone$, we set 
\[
L^2(\bG) = L^2(\mathbb{Z}) \otimes L^2(I_q) \otimes L^2(I_q),
\]
where each tensor component is understood with respect to the counting measure. It has a canonical basis $f_{m,p,t}$, with $m \in \mathbb{Z}, p,t \in I_q$.

The quantum group $\SUone$ was established in the Kustermans-Vaes setting by Koelink and Kustermans \cite{KoeKus}. Its Plancherel decomposition was obtained by Koelink, Kustermans and Groenevelt in \cite{GroKoeKus}. It is worth mentioning that $\SUone$ is constructed by first defining its multiplicative unitary and then reconstructing a von Neumann algebraic quantum group.
%Note that the notation  $\SUone$ is $\widetilde{SU_q(1,1)}$ in \cite{KoeKus}, since its construction is motivated by the normaliser of $SU(1,1)$ in $SL(2, \mathbb{C})$.  

Theorem-Definition \ref{Dfn=DfnSUone} defines $\SUone$. Its complete definition is rather involved, while in our proofs, we only use the Plancherel decomposition of its multiplicative unitary. Therefore, we  {\it define} $\SUone$ by  means of this decomposition. \eqref{Item=Principal} defines the principal series corepresentations, while \eqref{Item=Discrete} defines the discrete series corepresentations. The only fact we use about the discrete series is that their matrix coefficients are analytic extensions of the matrix coefficients of the principal series. This fact can be derived from a comparison of the actions of the generators of the universal enveloping Lie algebra \cite[Lemma 10.1 and Eqn. (92)]{GroKoeKus} on the representation spaces. 

Theorem-Definition \ref{Dfn=DfnSUone} uses the $C$-function defined in \cite[Lemma 9.4]{GroKoeKus}. In our proofs we compute special values of this functions and give further references. 

For direct integrals \cite{DixVna} is the standard reference.

\begin{thmdfn}\label{Dfn=DfnSUone}
Let $0 < q < 1$. Then, we define $\bG = \SUone$ as follows. 
\begin{enumerate}
\item\label{Item=Principal} For every $x \in [-1, 1], p \in q^{\mathbb{Z}}$, let 
\begin{equation}\label{Eqn=HSpace}
\cH_{p,x} = \spa\{  e^{\ep, \eta}_m(p,x) \mid  \ep, \eta \in \{+, -\}, m \in \textrm{$\frac{1}{2}$} \mathbb{Z}\}.
\end{equation}
where the vectors $e^{\ep, \eta}_m(p,x)$ are by definition orthonormal. 
There exists a unitary operator $W_{p,x} \in B(L^2(\bG) \otimes \cH_{p,x })$ determined by,
\begin{equation}\label{Eqn-CCoefficient}
\begin{split}
&(\id \otimes \omega_{e^{\ep,\eta}_m(p,x), e^{\ep',\eta'}_{m'}(p,x)})\left(W_{p,x} \right) f_{m_0, p_0, t_0}   \\
= &    C(\eta\ep x;m',\ep',\eta';\ep\ep'\vert p_0\vert
p^{-1}q^{-m-m'},p_0,m-m') \\
& \times \quad \delta_{sgn(p_0),\eta\eta'} f_{m_0 - m +
m',\ep\ep' \vert p_0\vert p^{-1} q^{-m-m'} , t_0},
\end{split}
\end{equation}
where the $C$-function is given in terms of basic hypergeometric series in \cite[Lemma 9.4]{GroKoeKus}. 
\item\label{Item=Discrete} For every $p \in q^{\mathbb{Z}}$, let the countable discrete set $\sigma_d(\Omega_p)$ be the discrete spectrum of the Casimir operator  \cite[Definition 4.5, Theorem 4.6]{GroKoeKus}  restricted to the space defined in \cite[Theorem 5.7]{GroKoeKus}. Let $\cH_{p,x}$ be the Hilbert space spanned by the orthonormal vectors $e^{\ep, \eta}_m(p,x)$ defined in \cite[Proposition 5.2]{GroKoeKus} (the notation is $\mathcal{L}_{p,x}$ instead of $\cH_{p,x}$ here and the span  of these vectors forms a subspace of \eqref{Eqn=HSpace}). There exists a unitary operator $W_{p,x} \in B(L^2(\bG)  \otimes \cH_{p,x })$ determined by  \eqref{Eqn-CCoefficient}.
\item  There exists a unique locally compact quantum group $\SUone$  with multiplicative unitary $W \in B(L^2(\bG)  \otimes L^2(\bG) )$ that is explicitly given in terms of the direct integral decomposition,
\begin{equation}\label{Eqn=Plancherel}
W = \bigoplus_{p \in q^\mathbb{Z}} \left( \int^\oplus_{[-1,1]} W_{p,x} \oplus \bigoplus_{x \in \sigma_d(\Omega_p)} W_{p,x} \right).
\end{equation}
\end{enumerate}
\end{thmdfn}
 
\begin{rmk}
For every $p \in q^\mathbb{Z}$ and $x \in [-1,1] \cup \sigma_d(\Omega_p)$ the operator $W_{p,x}$ defined in Theorem \ref{Dfn=DfnSUone} is a unitary {\it corepresentation} of $\SUone$. See, \cite[Proposition 5.2, Lemma 10.8]{GroKoeKus}. The corepresentations are not mutually inequivalent. 
\end{rmk}
\begin{rmk}
In \eqref{Eqn=Plancherel}, the measure on $[-1,1]$ is understood as the Askey-Wilson measure. For our purposes we need only  that this is a measure equivalent to the Lebesgue measure. 
\end{rmk}
\begin{rmk}\label{Rmk=TypeI}
It follows that the von Neumann algebra of the dual of $\SUone$ as a (proper) subalgebra of: 
\[
\bigoplus_{p \in q^\mathbb{Z}} \left( \int^\oplus_{[-1,1]} B(\cH_{p,x}) \oplus \bigoplus_{x \in \sigma_d(\Omega_p)} B(\cH_{p,x}) \right).
\]
(In fact in \cite[Proposition B.2]{CasKoe} it is proved that it is of type I.) It follows that finite linear combinations of inner product functionals with respect to vectors of the form,
\[
\int^\oplus_{[-1, 1]\cup \sigma_d(\Omega_p)} g(x) e^{\ep, \eta}_m(p,x) dx \in L^2(\bG), \quad \ep, \eta \in \{ -,+\}, p \in q^{\mathbb{Z}}, m \in \mathbb{Z},
\]
and $g$ a square integrable function on $[-1, 1]\cup \sigma_d(\Omega_p)$, form a separating set of functionals for the dual von Neumann algebra of $\SUone$. 
\end{rmk}

\section{Coamenability} \label{Sect=Coamenability}

The main result of this section is that $SU_q(1,1)_{{\rm ext}}$ is coamenable. There are various equivalent notions of coamenability, see  \cite{BedTus}.  We recall the following four. Whereas \eqref{Item=CoamI} is commonly used in the literature, we prove \eqref{Item=CoamII} of Lemma \ref{Lem=CoamenabilityEquivalence} in the main theorem of this section. \eqref{Item=CoamIV} was used in the proof of Lemma \ref{Lem=L1Approx}. 
For $\xi, \eta \in L^2(\bG)$, we set $\omega_{\xi, \eta} \in L^1(\bG)$ by  $\omega_{\xi, \eta}(x) = \langle x \xi, \eta\rangle$. Furthermore, $\omega_\xi = \omega_{\xi, \xi}$. 
 
\begin{lem}\label{Lem=CoamenabilityEquivalence}
Let $\mathbb{G}$ be a locally compact quantum group. The following are equivalent:
\begin{enumerate}
\item\label{Item=CoamI} $L^1(\mathbb{G})$ has a bounded approximate identity. That is, there exists a bounded net $\{ \omega_i \}_i$ in $L^1(\mathbb{G})$ such that for every $\omega \in L^1(\mathbb{G})$ we have $\Vert \omega_i \ast \omega - \omega \Vert_{L^1(\mathbb{G})} \rightarrow 1$.
\item\label{Item=CoamII}  There exists a net of unit vectors $\{ \xi_i \}_i$ in $L^2(\mathbb{G})$ such that $(\omega_{\xi_i} \otimes id)(W) \rightarrow 1$ in the $\sigma$-weak topology of $L^\infty(\hat{\mathbb{G}})$. 
\item\label{Item=CoamIII}  There exists a net of unit vectors $\{ \xi_i \}_i$ in $L^2(\mathbb{G})$ such that $(\omega_{\xi_i} \otimes id)(W^\ast) \rightarrow 1$ in the topology induced by a separating set of vectors in $L^1(\bG)$.
\item\label{Item=CoamIV} There exists a state $\epsilon: C_0(\bG) \rightarrow \mathbb{C}$ such that $(\epsilon \otimes id)(W) = 1$. 
\end{enumerate}
If $\mathbb{G}$ satisfies these criteria, then it is called {\it coamenable}. The notation is consistent in the sense that the nets $\{ \xi_i \}_i$ in \eqref{Item=CoamII} and \eqref{Item=CoamIII} can be taken the same. 
\end{lem}
\begin{proof}
\eqref{Item=CoamI} if and only if \eqref{Item=CoamII} if and only if \eqref{Item=CoamIV} is proved in  \cite{BedTus}.  \eqref{Item=CoamII} if and only if  \eqref{Item=CoamIII} follows from a standard convexity argument and the fact that $(\omega_{\xi_i} \otimes id)(W^\ast) = (\omega_{\xi_i} \otimes id)(W)^\ast$.
\end{proof}
 
\begin{thm}\label{Thm=Coamenability}
Let $\bG = \SUone$. For $n \in \mathbb{N}, p_1 \in q^{\mathbb{Z}}$ we define the unit vector, 
\[
\xi_{n,p_1} = \frac{1}{\sqrt{2n+1}}  \sum_{k=-n}^n f_{0, p_1 q^{k}, 1} \in L^2(\mathbb{G}).
\]
Let $I = \mathbb{N} \times q^\mathbb{Z}$ and   define a net structure on $I$ by saying that for $(n, p_1), (n', p_1') \in I$  we have $(n, p_1) \leq  (n', p_1')$ if and only if $n \leq n'$ and $p_1q^n  \leq p_1' q^{n'}$. 
Then, 
$
(\omega_{\xi_{n,p_1}} \otimes id)(W)\rightarrow 1 
$
$\sigma$-weakly. That is, $\mathbb{G}$ is coamenable. 
\end{thm}
\begin{proof}
We start with computing explicit matrix coefficients of $W^\ast$. We take $p \in q^{\mathbb{Z}}$ and let $x \in [-1, 1] \cup \sigma_d(\Omega_p)$ so that $W_{p,x}$ is a corepresentation weakly contained in $W$ (i.e. it occurs on the Plancherel decomposition of $W$).  For  $p_1 \in q^\mathbb{Z}, k \in \mathbb{N}, m \in \mathbb{Z}, \ep, \eta \in \{-1, 1\}$, we find using \cite[Lemma 10.7]{GroKoeKus} and its short proof, that
\[
\begin{split}
& (\omega_{f_{0, p_1, 1}, f_{0, p_1 q^k, 1} } \otimes \id )(W^\ast_{p,x})  e_m^{\ep, \eta}(p,x)   \\
= \: & \delta_{2m, k-\chi(p)} C(\ep\eta x; m, \ep, \eta; p_1, p_1 q^k, 0) e_m^{\ep, \eta}(p,x).  
\end{split}
\]
To prevent tedious notation, we concentrate on $p = 1$ only. The reader may verify that for different $p\in q^{\mathbb{Z}}$ one gets similar (shifted) expressions. 
In case $ -n \leq m \leq 0$, 
\[ 
(\omega_{\xi_{n,p_1}} \otimes id)(W^\ast_{1,x}) e_{m }^{\ep, \eta}(1,x)  
=    \frac{1}{2n+1}  \sum_{p_0 =  p_1 q^{n+2m} }^{ p_1q^{-n}}\!\!\! \!\! C(\ep\eta x; m, \ep, \eta; p_0, p_0 q^{2m}, 0) e_{m }^{\ep, \eta}(1,x),
\]
and in case $n \geq m \geq 0$,
\[
(\omega_{\xi_{n,p_1}} \otimes id)(W^\ast_{1,x}) e_m^{\ep, \eta}(p,x) =
 \frac{1}{2n+1} \!\!  \sum_{p_0 = p_1 q^{n-2m}}^{p_1q^{-n}}\!\!\! \!\! C(\ep\eta x; m, \ep, \eta; p_0, p_0 q^{2m}, 0) e_m^{\ep, \eta}(p,x),
\]
where the sums over $p_0$ take  values in $q^{\mathbb{Z}}$ and hence, there are exactly $2(n-\vert m\vert )+1$ summands. Since $(2(n-\vert m\vert )+1)/(2n+1) \rightarrow 1$ as $n \rightarrow \infty$, it suffices to prove that 
\[
C(\ep\eta x; m, \ep, \eta; p_1, p_1 q^{2m}, 0) \rightarrow 1, \qquad \textrm{ as } p_1 \rightarrow \infty,
\]
uniformly in $x$ on compact sets of $[-1, 1] \cup \sigma_d(\Omega_1)$, since by the Plancherel decomposition, this entails \eqref{Item=CoamIII} of Lemma \ref{Lem=CoamenabilityEquivalence}. Indeed, as the separating set of functionals, one can for example take direct integrals over a compact index of finite linear combinations of inner product functionals with respect to the vectors $e^{\ep,\eta}_m(p,x)$ (see also Remark \ref{Rmk=TypeI}).  

Let $\lambda \in \mathbb{C}$ be such that $\mu(\lambda) = x$. Then, see \cite[Lemma 9.1]{GroKoeKus} for the $S$-function, 
\begin{equation}\label{Eqn=CoamenabilityCompI}
\begin{split}
& C(\ep \eta x; m, \ep, \eta, p_1, p_1q^{2m}, 0)\\
 = & S(-\lambda; p_1, p_1 q^{2m}, 0) \\
= & p_1^2 q^k \nu(p_1)\nu(p_1 q^{2m}) c_q^2 \sqrt{(-\kappa(p_1), -\kappa(p_1 q^{2m}); q^2  )_\infty } \\
& \times \:\frac{(q^2, -q^2/\kappa(p_1 q^{2m}), \lambda q^3/p_1^2q^{2m}, p_1^2 q^{{2m}-1}/\lambda, -q^{1-{2m}}/\lambda; q^2)_\infty  )}{(   p_1^2 q^{{2m}-1}/\lambda, \lambda q^3/p_1^2q^{2m}, -q^{1-{2m}}/\lambda; q^2)_\infty } \\
& \times \:(q^2; q^2)_\infty  \: _2 \! \varphi_1\left( \begin{array}{c} -q^{1+{2m}}/\lambda,  -\lambda q^{1+{2m}} \\ q^2 \end{array};q^2, -q^2/\kappa(p_1q^{2m}) \right)\\
= & p_1^2 q^{2m} \nu(p_1)\nu(p_1 q^{2m}) c_q^2 \sqrt{(-\kappa(p_1), -\kappa(p_1 q^{2m}); q^2  )_\infty } \\
& \times \: (q^2, -q^2/\kappa(p_1 q^{2m}); q^2)_\infty  )  
  (q^2; q^2)_\infty  \\
& \times \: _2 \! \varphi_1\left( \begin{array}{c} -q^{1+{2m}}/\lambda,  -\lambda q^{1+{2m}} \\ q^2 \end{array};q^2, -q^2/\kappa(p_1q^{2m}) \right)
\end{split}
\end{equation}
As $p_1 \rightarrow \infty$ we have that,
\[
_2 \! \varphi_1\left( \begin{array}{c} -q^{1+2m}/\lambda,  -\lambda q^{1+2m} \\ q^2 \end{array};q^2, -q^2/\kappa(p_1q^{2m}) \right)\rightarrow 1,
\]
uniformly on compact sets in $\lambda$ (such that $\mu(\lambda) \in [-1,1] \cup \sigma_d(\Omega_p)$). (That the convergence is uniform is well known. Alternatively, it can be derived from the Arzela-Ascoli theorem, which implies that it is enough to have a bounded sequence of locally analytic functions that converges pointwise).  Hence, it remains to check that the coefficient of this function in \eqref{Eqn=CoamenabilityCompI} converges to 1 as $p_1 \rightarrow \infty$. We find, putting $p_0 = q^l$ in the third equality,
\begin{equation} \label{Eqn=CoamenabilityCompII}
\begin{split}
& p_1^2 q^{2m} \nu(p_1)\nu(p_1 q^{2m}) c_q^2 \sqrt{(-\kappa(p_1), -\kappa(p_1 q^{2m}); q^2  )_\infty }  \\
& \times \:    (q^2, -q^2/\kappa(p_1 q^{2m}); q^2)_\infty  )  
  (q^2; q^2)_\infty \\
= & c_q^2 (q^2; q^2)_\infty   p_1^2 q^{2m} \nu(p_1)\nu(p_1 q^{2m}) \sqrt{(  -\kappa(p_1 q^{2m}), -q^2/ \kappa(p_1 q^{2m}) ; q^2  )_\infty }   \\
&\times \: 
  \sqrt{(  -\kappa(p_1 ), -q^2/ \kappa(p_1 q^{2m}) ; q^2  )_\infty } \\
= &   c_q^2 (q^2; q^2)_\infty   p_1^2 q^{2m} \nu(p_1)\nu(p_1 q^{2m}) \sqrt{(  -\kappa(p_1 q^{2m}), -q^2/ \kappa(p_1 q^{2m}) ; q^2  )_\infty } \\
 & \times \:    
  \sqrt{(  -\kappa(p_1 ), -q^2/ \kappa(p_1 ) ; q^2  )_\infty  ( -q^2/p_1^2q^{4m};q^2 )_{2m} } \\
= &   c_q^2 (q^2; q^2)_\infty   p_1^2 q^{2m} \nu(p_1)\nu(p_1 q^{2m})  \sqrt{ q^{-({2m}+l)({2m}+l-1)}  ( -1, -q^2 ; q^2  )_\infty }    \\
 & \times \: 
  \sqrt{ q^{-l(l-1)} (- 1, -q^2 ; q^2  )_\infty  ( -q^2/p_1^2q^{4m};q^2 )_{2m} } \\
= &  c_q^2 q^2 (q^2; q^2)_\infty^2 (-1,-q^2; q^2)_\infty \sqrt{  ( -q^2/p_1^2q^{4m};q^2 )_{2m} }\\
= &    \sqrt{ ( -q^2/p_1^2q^{4m};q^2 )_{2m} }
\end{split}
\end{equation}
Here, the first and second equality are elementary rearrangements of the terms, the third equality follows from the $\theta$-product identity, the fourth equality follows from an elementary computation using the replacement $p_1 = q^l$, the last equality follows from the definition of $c_q$ and the $\theta$-product identity. If $p_1 \rightarrow \infty$, we find that $  \sqrt{ ( -q^2/p_1^2q^{4m};q^2 )_{2m} } \rightarrow 1$ and the theorem follows. 
\end{proof}

\begin{rmk}\label{Rmk=Reflexion}
It was pointed out to the author that reflexion of a quantum group (see \cite{ComI}) preserves coamenability. This can be proved from an unpublished result due to De Commer \cite[Section 7.6]{ComPhD}. Here, a Morita equivalence of both the universal and reduced C$^\ast$-algebraic quantum group and its reflection is established. Coamenability of $\SUone$ follows then from \cite{ComII} and the well known fact that $SU_q(2)$ is coamenable.  Since these results from \cite{ComPhD} are unpublished we give a direct proof here. It also gives the counit explicitly.  
\end{rmk}
\section{Spherical functions}\label{Sect=Spherical}
 
Throughout the section we put $\bG = \SUone$. We construct matrix coefficients of $\bG$ that form an approximate identity of $C_0(\bG)$. For convenience, we set the corepresentations, where $z = ib$ with $0 \leq b \leq -\frac{\pi}{ \log(q)}$ (natural logarithm), 
\[
V_z = W_{1, \mu(q^z)}, \quad \cH_{z} = \cH_{1,\mu(q^z) },   
\]
So $V_z \in L^\infty(\bG) \otimes B(\cH_z)$. We define the unit vector,
\begin{equation}\label{Eqn=State}
 f_z = \frac{1}{2} \sqrt{2} \left( e^{+,-}_{0}(1,\mu(q^z)) + e^{-,+}_{0}(1,\mu(q^z)) \right) \in \cH_{z},  
\end{equation}
and set, 
\begin{equation}\label{Eqn=AzFunction}
\begin{split}
a_z  & =    (\id \otimes \omega_{f_z, f_z})\left( V_z \right). 
\end{split}
\end{equation}
These are spherical matrix coefficients associated with the irreducible components of $V_z$, see \cite{Cas}. 

 It follows from \eqref{Eqn-CCoefficient} that $f_{m_0,p_0,t_0}$ is an eigenvector for $a_z$, where the eigenvalue is independent of $m_0$ and $t_0$. We let $a_z(p_0)$ be the eigenvalue of the vector $f_{m_0, p_0, t_0}$ for $a_z$. We will regard $a_z$ as a function on $I_q$ as well as an operator in $L^\infty(\bG)$. 

\begin{prop}\label{Prop=Analytic}\label{Thm=Approx}
We collect the following properties for $a_z$.  
\begin{enumerate}
\item\label{Item=AnalyticI} There is a simply connected neighbourhood $\cG$ of $i \mathbb{R} \cup [0, 1)$ such that for every $p_0 \in I_q$, the function $z\mapsto a_z(p_0)$   extends analytically to $\cG$. Moreover, for every $\alpha > 1$ we can choose $\cG = \cG_\alpha$ such that  $a_z$ is the matrix coefficient of a (possibly  non-unitary) invertible corepresentation $V_z$ of $\bG$ with, 
\[
  \Vert V_z \Vert \leq \alpha \quad \textrm{ and } \quad \Vert V_z^{-1} \Vert \leq  \alpha.
\]
\item\label{Item=AnalyticII} For every $z \in i \mathbb{R}$ we have $a_{z+ \frac{2i\pi}{\log(q)}} = a_{z}$. 
\item\label{Item=AnalyticIII} For every $p_0 \in I_q $,
\[
\lim_{z \rightarrow 1} a_z(p_0) = 1.
\]  
 Moreover, this convergence is uniform  on $I_q \cap [1, \infty)$. 
\end{enumerate}
\end{prop}
%\begin{rmk}
%Since the $C$-function is continuous at $p_0 = 0$, see Remark \ref{Rmk=Continuity}, we could put $a_z(0)$ as the limit $p_0 \rightarrow 0$ of $a_z(p_0)$. \eqref{Item=AnalyticIII} and \eqref{Item=AnalyticIV}   then yield that $a_z$ converges to 1 uniformly on compact sets of $I_q \cup \{ 0 \}$ as $z \rightarrow 1$. However, we found it more convenient to seperate cases in order to deal easily with the phase factor of Lemma \ref{Lem=Phase}.  
%\end{rmk}
\begin{proof} 
\eqref{Item=AnalyticI} The first claim was  observed in \cite[Section 10.3]{GroKoeKus}. That is, that $z\mapsto a_z(p_0)$  extends analytically to a neighbourhood of  $i [0, -\frac{\pi }{\log(q)}] \cup [0, 1)$. The fact that the analytic domain $\cG$ can be extended to a neighbourhood of the whole imaginary axis, is a consequence of the fact that 
\[
\mu(q^{i t \log(q)}) =  \mu(q^{i (2\pi - t) \log(q)}) = \mu(q^{i (t+2\pi) \log(q)}),
\]
and using the Schwartz reflection principle. To apply this principle, we need to check that $a_z(p_0)$  takes real values for $z \in \frac{i \pi\mathbb{Z}}{ \log(q)} + (-\varepsilon, \varepsilon)$ for certain $\varepsilon >0$. But this follows from the explicit expressions of $a_z(p_0)$ which where computed along the proof of Proposition \ref{Prop=TwoPhiOneApproximation}. 

Recall that for a Hilbert space $\cH$, the  invertible operators form an open subset of $B(\cH)$. Since $z \mapsto V_z$ extends analytically to a neighbourhood of $i \mathbb{R} \cup [0, 1)$, and for $z \in i \mathbb{R} \cup [0, 1)$ the corepresentation $V_z$ is unitary \cite[Section 10.3]{GroKoeKus}, we may choose the neighbourhood $\cG_\alpha$ small; i.e. such that  $\Vert V_z \Vert \leq \alpha$ and $\Vert V_z^{-1} \Vert \leq  \alpha$.

 \eqref{Item=AnalyticII} The property follows directly from  the symmetry argument in \eqref{Item=AnalyticI}.

\eqref{Item=AnalyticIII} We postpone this proof to the appendix, see Proposition \ref{Prop=TwoPhiOneApproximation}.  

\end{proof}
 
\begin{rmk}
It is unknown what the exact domain of  $\cG$ in Proposition \ref{Prop=Analytic} is, we merely know its existence. Let us also indicate that as $q \rightarrow 1$, the periodicity \eqref{Item=AnalyticII} tends to infinity, resulting in the classical limit. This was already observed in \cite{MasudaEtAl}.
\end{rmk}

\begin{thm}\label{Thm=ApproxSU} 
Let $\bG = \SUone$, let $a_z$ be the matrix coefficient  defined by \eqref{Eqn=AzFunction}  and its analytic extension to $\cG_\alpha$ as in Proposition \ref{Prop=Analytic}.    For every $c \in C_0(\bG)$ we have $\Vert a_z c - c \Vert_{L^\infty(\bG)} \rightarrow 0$ as $z \rightarrow 1$ in $\cG_\alpha$. 
\end{thm}
\begin{proof}  
Recall that $L^\infty(\bG)$ is given by $L^\infty(\mathbb{T}) \otimes B(L^2(I_q))$ which acts on the first two legs of  $L^2(\bG) = L^2(\mathbb{Z}) \otimes L^2(I_q) \otimes L^2(I_q)$  by identifying $L^2(\mathbb{Z})$ with $L^2(\mathbb{T})$ under the Fourier transform, c.f. \cite[Lemma 2.4]{KoeKus}.  
 In this representation $a_z$ corresponds to an operator in $1 \otimes B(L^2(I_q))$ and in fact the canonical basis $1 \otimes \delta_{p_0} \in L^2(\mathbb{T}) \otimes L^2(I_q)$ with $p_0 \in I_q$ forms a complete set of eigenvectors of $a_z$ with respective eigenvalues $a_z(p_0)$. 

For $i \in \{ 1, 2\}$ we define,
\[
v_i = \int^\oplus_{[-1, 1]\cup \sigma_d(\Omega_p)} g_i(x) e^{\ep_i, \eta_i}_{m_i}(p_i,x) dx \in L^2(\bG), 
\]
for some $\ep_i, \eta_i \in \{ -,+\}, p_i \in q^{\mathbb{Z}}, m_i \in \mathbb{Z}$ and  $g_i$ a square integrable function on $[-1, 1]\cup \sigma_d(\Omega_p)$. Set $x = (\id \otimes \omega_{v_1, w_2})(W)$. Then, $x \in C_0(\bG)$ and in fact the linear span of such elements $x$ forms a norm dense subset of $C_0(\bG)$ see Remark \ref{Rmk=TypeI}.

Fix $m_0 \in \mathbb{Z}$,    $t_0   \in I_q$. From \eqref{Eqn-CCoefficient} one sees that there is at most one $m_2 \in \mathbb{Z}$ and $p_2 \in q^{\mathbb{Z}}$   and $\pm$ either $+$ or $-$ such that the following function is non-zero:
\begin{equation}\label{Eqn=HelpfulCoef}
\Phi_x: I_q \rightarrow \mathbb{C}: p_0 \mapsto \langle x f_{m_0, p_0, t_0}, f_{m_0+m_2, \pm p_0p_2, t_0} \rangle.
\end{equation}
Using that the Haar weight,
\[
{\rm Tr}(\: \cdot \:) \otimes \sum_{p_0 \in I_q}  p_0^{-2} \langle \: \cdot \: \delta_{p_0}, \delta_{p_0}\rangle,
\]
restricts to a semi-finite weight on $C_0(\bG)$ we must have $\Phi_x(p_0) \rightarrow 0$ as $p_0 \rightarrow 0$, since else $x$ cannot be approximated in norm by square integrable elements of $C_0(\bG)$.

Consider the projections $P_0: L^2(I_q) \rightarrow L^2(I_q \cup (-1, 1))$ and $P_1: L^2(I_q) \rightarrow L^2(I_q \cup [1, \infty))$. We decompose $x = P_0 x + P_1 x$ where $P_0 x, P_1 x \in L^\infty(\bG)$. We need to prove that, 
\[
\Vert a_z P_0 x - P_0 x \Vert_{L^\infty(\bG)} \rightarrow 0 \textrm{ and } \Vert a_z P_1x - P_1x \Vert_{L^\infty(\bG)} \rightarrow 0.
\] 
The right convergence follows since $\Vert a_z P_1 - P_1\Vert_{L^\infty(\bG)} \rightarrow 0$ by Proposition \ref{Prop=Analytic}.  For the left convergence, it follows from \eqref{Eqn-CCoefficient} that it suffices to show that $\Phi_{a_z x - x}(p_0) \rightarrow 0$ as $z \rightarrow 1$ uniformly for $p_0 \in (-1, 1) \cap I_q$. But this follows from Proposition \ref{Prop=Analytic}, the fact that $\Phi_x(p_0)$ tends to zero as $p_0 \rightarrow 0$ and the fact that $a_z(p_0)$ is bounded in $p_0$.

\end{proof}

\begin{comment}
\begin{dfn}
We call a possibly non-unitary corepresentation $U \in M \otimes B(\cH_U)$ {\it invertible} if it is invertible as an operator.
\end{dfn}

The operators $W_{1, \mu(q^z)}$ indexed by $z$ in $A :=  i\mathbb{R} \cup  [0, 1]$ admit an analytic extension to the neighbourhood $G$ of Proposition \ref{Prop=Analytic}. For $z \in G$, $W_{1, \mu(q^z)}$ is again a corepresentation. Indeed, since both $(\Delta \otimes id)(W_{1, \mu(q^z)})$ and $(W_{1, \mu(q^z)})_{13} (W_{1, \mu(q^z)})_{23}$ are analytic extensions from $A$ to $G$, they must agree. Moreover, since the invertible operators form an open set, we may choose $G$ in such a way that for every $z \in G$, the corepresentation $W_{1, \mu(q^z)}$ is invertible.

\begin{rmk}
For a group $H$, a representation $\pi$ is called {\it bounded} if $\sup_{h \in H} \Vert \pi(h) \Vert < \infty$. Since $\pi(h^{-1}) = \pi(h)^{-1}$, such a representation is automatically invertible and vice versa. For general quantum groups, REFER TO DAWS WORK
\end{rmk}
\end{comment}

\section{Weak amenability}\label{Sect=WeakAmenability}

Recall that we defined weak amenability in Section \ref{Sect=CompactaApproximation}. Here, we prove that $\SUqone$ is weakly amenable.  From this point, the proof is essentially the same as \cite[Theorem 3.7]{CanHaa}. The necessary modifications to lift the arguments from groups to quantum groups are presented in this section. In particular, we use coamenability as established in Section \ref{Sect=Coamenability} to apply Theorem \ref{Thm=CompactaApproximation}. 

\vspace{0.3cm}

Let us mention that in \cite{CowHaa} it is proved  that every real rank one simple Lie group $G$ with finite center is weakly amenable with Cowling-Haagerup constant depending on the local isomorphism class of $G$. If the rank is greater than 1, $G$ is not weakly amenable \cite{Haa}. 

The most important examples of weakly amenable quantum groups come from Freslon's result \cite{Freslon}, showing that the free orthogonal and free unitary quantum groups of Kac type are weakly amenable.  

\vspace{0.3cm}

The following lemma is the quantum group analogue of \cite[Theorem 2.2]{CanHaa}. It follows directly from \cite[Corollary 4.8]{BraDawSam}. See also \cite[Proposition 4.1]{Daws}.

\begin{lem}\label{Lem=FellAbsorption}
Let $\bG$ be a locally compact quantum group. Let $U \in L^\infty(\mathbb{G}) \otimes B(\cH_U)$ be an invertible corepresentation of $\bG$. Let $\omega \in B(\cH_U)_\ast$ and set $a = (id \otimes \omega)(U)^\ast$. Then, $a \in \MCB(\bG)$. Moreover, $\Vert a \Vert_{\MCB(\bG)} \leq \Vert U \Vert \Vert U^{-1} \Vert \Vert \omega \Vert$. 
\end{lem}
\begin{proof}
\begin{comment}
Since, $W_{12}^\ast U_{23} W_{12} = (\Delta \otimes id)(U) = U_{13}U_{23}$, we have $W_{12} U_{13} = U_{23} W_{12} U_{23}^{-1}$. Let $\theta \in L^1(\hat{\mathbb{G}})$. Then, 
\[
(id \otimes \theta)(W) a = (id \otimes \theta \otimes \omega)(W_{12}U_{13}) = (id \otimes \theta \otimes \omega)( U_{23} W_{12} U_{23}^{-1}),
\] 
is a matrix coefficient of $W \otimes 1$ and hence an element of $A(\bG)$. Thus, $a \in MA(\bG)$. In fact, $a \in \MCB(\bG)$ as can be seen from a standard trick by tensoring $\bG$ with $SU(2)$, see the proof of \cite[Theorem 1.6]{CanHaa}. 
\end{comment}
The slice maps $\pi: L^1(\bG) \rightarrow L^\infty(\bG): \omega \mapsto (\omega \otimes id)(U)$ and $\check{\pi}: L^1(\bG) \rightarrow L^\infty(\bG): \omega \mapsto (\omega \otimes id)(U^{-1})$ are completely bounded and $\Vert \pi \Vert_{\CB} = \Vert U \Vert, \Vert \check{\pi} \Vert_{\CB} = \Vert U^{-1} \Vert$. Indeed, recall from \cite[Chapter 7]{EffRua} that,
\[
\CB(L^1(\bG), B(\cH)) \simeq  (L^1(\bG) \hat{\otimes} \cT(\cH))^\ast \simeq L^\infty(\bG) \otimes B(\cH), 
\]
where $\hat{\otimes}$ is the operator space projective tensor product. Moreover, the correspondence is given such that $\pi$ corresponds to $U$ and $\check{\pi}$ to $U^{-1}$.  
 Hence, \cite[Corollary 4.8]{BraDawSam} yields that $a \in \MCB(\bG)$ with bound $\Vert a \Vert_{\MCB(\bG)} \leq \Vert U \Vert \Vert U^{-1} \Vert \Vert \omega \Vert$.
\end{proof}
\begin{cor}\label{Cor=SUqOneMultipliers}
Let $\bG = \SUqone$, $\alpha > 1$ and let  $a_z = (\id \otimes \omega_{f_z})(V_z) \in L^\infty(\bG)$ with $z \in \cG_\alpha$ be as in Proposition  \ref{Prop=Analytic}. Then, $a_z  \in \MCB(\bG)$ and $\Vert a_z  \Vert_{\MCB(\bG)} \leq \alpha^2$. 
\end{cor}
\begin{proof}
By Lemma \ref{Lem=FellAbsorption} applied to the contragredient corepresentation of $V_z$, we see that $a_z \in \MCB(\bG)$ and   $\Vert a_z  \Vert_{\MCB(\bG)} \leq \alpha^2$. 
\begin{comment}
Now, define the anti-linear map, 
\[
K: L^1(\hat{\mathbb{G}}) \rightarrow L^1(\hat{\mathbb{G}}): \omega \mapsto K(\omega),  
\]
where $K(\omega)(x) = \overline{\omega(x)}, x \in L^\infty(\hat{\bG})$. $K$ is a completely isometric anti-linear map. Furthermore, the explicit form of $a_z$ yields that,
\[
K\left(\hat{\lambda}^{-1}(  a_z ^\ast \hat{\lambda}( K(\omega) ) )\right) =  \hat{\lambda}^{-1}( a_z   \hat{\lambda}(\omega)), 
\]
so that $a_z \in \MCB(\bG)$ with the same completely bounded norm as  $ a_z ^\ast \in \MCB(\bG)$.
\end{comment}
\end{proof}
\begin{lem}\label{Lem=ComP} 
Let $\bG = \SUqone$, $\alpha > 1$ and let  $a_z  \in L^\infty(\bG)$ with $z \in \cG_\alpha$ be as in Proposition  \ref{Prop=Analytic}. For $z \in [0,1]$ the multiplier $a_z \in \MCB(\bG)$ is completely positive. 
\end{lem}
\begin{proof}
Recall that $a_z = (\id \otimes \omega_{f_z})(V_z)$ and that for  $z \in [0,1]$ the corepresentation $V_{z}$ is unitary (see \cite[Section 10.3]{GroKoeKus}). For every $\omega \in L^1(\bG)^\sharp$ it follows that 
\[
(\omega^\ast \otimes \omega)(V_z) = \left((\omega \otimes  \id )(V_z)\right)^\ast (\omega \otimes  \id )(V_z ) \geq 0.
\]
 Using \cite[Proposition 20]{DawSal} for coamenable quantum groups, we see that $a_z$ is in fact completely positive.  
\end{proof}

Along the following proof we use the well-known fact that a function $f: \mathbb{C} \rightarrow \mathcal{X}$ with $\mathcal{X}$ a Banach space is norm analytic if and only if for a separating set of functionals $\beta \in \mathcal{X}^\ast$ the function $z \mapsto \langle f, \beta \rangle_{\mathcal{X}, \mathcal{X}^\ast}$ is analytic. 

\begin{thm}
Let $\bG =  SU_q(1,1)_{{\rm ext}}$. There exists a net $\{ a_i \}$ in $A(\bG)$ such that for every $b \in A(\bG)$ we have $\Vert a_i b - b \Vert_{A(\bG)} \rightarrow 0$. Moreover, the net can be choosen such that,
\[
\limsup_{i \in I} \Vert a_i \Vert_{\MCB(\mathbb{G})} = 1.
\]
That is, $\bG$ is weakly amenable with Cowling-Haagerup constant $\Lambda(\bG) = 1$. 
\end{thm}
\begin{proof}[Proof (following Theorem 3.7 of \cite{CanHaa}).]
%By Theorem \ref{Thm=Approx}, there exists an open neighbourhood $\cG$ of $i \mathbb{R} \cup [0, 1)$ such that: 
%\begin{enumerate}
%\item For every $z \in \cG$, the corepresentation $W_{1, \mu(q^{z})}$ is invertible and we have bounds:
%\[
%\Vert W_{1, \mu(q^{z})} \Vert < 2, \qquad \Vert W_{1, \mu(q^{z})} \Vert < 2.
%\] 
%\item The map $\cG \mapsto L^\infty(\bG) \otimes B(\cH_0): z \mapsto  W_{1, \mu(q^{z})}$ is analytic. 
%\end{enumerate}
Let $\alpha > 1$ and let $\cG_\alpha$ and $a_z$ be as in Proposition \ref{Prop=Analytic}. By Corollary \ref{Cor=SUqOneMultipliers}   we find that $a_z  \in \MCB(\bG)$ with  $\Vert a_z  \Vert_{\MCB(\bG)} \leq  \alpha^2$. 
%We denote the associated multiplier by,
%\[
%L_z: L^1(\hat{\bG}) \rightarrow  L^1(\hat{\bG}): \omega \mapsto  \hat{\lambda}^{-1}( a_z  \hat{\lambda}(\omega)).
%\]
 
Define a path $\gamma_0: \mathbb{R} \rightarrow \cG_\alpha: s \mapsto is$. For $k \in \mathbb{N}$, we fix a (continuous) path $\gamma_k: \mathbb{R} \rightarrow \cG_\alpha$ that satisfies the following two criteria: (1) $\gamma_k(0) = 1- \frac{1}{k}$, (2) $\vert \gamma_1(s) - \gamma_0(s) \vert < 37$ for all $s \in \mathbb{R}$.  
Define,
\[
b_{k,n} = \sqrt{\frac{n}{\pi}} \int_{\gamma_k}  e^{-n (z-1+\frac{1}{k})^2} a_{z} dz.
\]
We claim that the integral exists as a Bochner integral in $\MCB(\bG)$. Indeed, for $\omega \in L^1(\bG)$ we see that $z \mapsto e^{-n (z-1+\frac{1}{k})^2}\omega( a_{z} )$ is analytic. Using \cite[Theorem 3.4]{HuNeuRua} (but then for left multipliers), this shows that  $z \mapsto e^{-n (z-1+\frac{1}{k})^2}  a_{z}  $ is analytic in $\MCB(\bG)$. Moreover, $b_{k,0} \in A(\bG)$. We also set,
\[
b_k = a_{1-1/k}, \qquad k \in \mathbb{N}.
\]
Recall that $\Vert a_z \Vert_{\MCB(\bG)} \leq  \alpha^2$ for every $z \in \cG$. It follows from Cauchy's theorem and assertion (2) on our path $\gamma_k$ that,
\begin{equation}\label{Eqn=CauchyTheorem}
b_{k,n} = \sqrt{\frac{n}{\pi}} \int_{\gamma_k}  e^{-n (z-1+\frac{1}{k})^2} a_{z} dz  = \sqrt{\frac{n}{\pi}} \int_{\gamma_0}  e^{-n (z-1+\frac{1}{k})^2} a_{z} dz = b_{k,0} \in A(\bG),
\end{equation}
where the integrals are understood within $\MCB(\bG)$. Here, the second equality of \eqref{Eqn=CauchyTheorem} is explained in full detail in \cite[p. 482]{CanHaa} and we leave the mutatis mutandis copy to the reader. Also, for every $k \in \mathbb{N}$, we find that as $n \rightarrow \infty$,
\[
\begin{split}
&\Vert b_{k,n} - b_k \Vert_{\MCB(\bG)} \\
= &
\Vert \sqrt{\frac{n}{\pi}} \int_{\gamma_k} e^{-n(z-1+\frac{1}{k})^2} a_z dz - a_{1-1/k} \Vert_{\MCB(\bG)} \\
\leq & \Vert  \sqrt{\frac{n}{\pi}} \int_{\gamma_k} e^{-n(z-1+\frac{1}{k})^2} a_z - a_{1-1/k}  dz  \Vert_{\MCB(\bG)} \\ & \: +\:   
\Vert  \left( \sqrt{\frac{n}{\pi}} \int_{\gamma_k} e^{-n(z-1+\frac{1}{k})^2}   dz -1 \right)a_{1-1/k}  \Vert_{\MCB(\bG)} \\
\leq  & \sqrt{\frac{n}{\pi}} \int_{\gamma_k}  \vert e^{-n (z-1+\frac{1}{k})^2}\vert\: \Vert a_{z} -a_{1-1/k} \Vert_{\MCB(\bG)}  dz \\ & \: +\: 
 \vert \sqrt{\frac{n}{\pi}} \int_{\gamma_k} e^{-n(z-1+\frac{1}{k})^2}   dz -1 \vert \Vert a_{1-1/k}  \Vert_{\MCB(\bG)} \\
\rightarrow &0.
\end{split}
\]
This implies that for every $c \in A(\bG)$ as $n \rightarrow \infty$,
\[
\Vert b_{k,n}  c - b_k c \Vert_{A(\bG)} \rightarrow 0. 
\]
Using Theorem  \ref{Thm=ApproxSU}  we see furthermore that for every $c \in C_0(\bG)$ we have,
\[
\Vert  b_k  c - c\Vert_{A(\bG)} \rightarrow 0. 
\]
Note that $b_k$ is a positive multiplier by Lemma \ref{Lem=ComP} and Theorem \ref{Thm=Coamenability} shows that $\bG$ is coamenable. Theorem \ref{Thm=CompactaApproximation} then concludes our proof. Since we could choose $\alpha>1$ arbitrary,  and $\Vert a_z \Vert_{\MCB(\bG)} \leq \alpha^2$, we find that $\Lambda(\bG) \leq 1$.  
\end{proof}

\section{Haagerup property}\label{Sect=Haagerup} 

Recently, the Haagerup property was introduced conceptually for locally compact quantum groups by Daws, Fima, Skalski and White, see \cite{DawFimSkaWhi}. Let us state the  definition that is most convenient for us. See \cite[Theorem 5.5.(iii)]{DawFimSkaWhi}. 

\begin{dfn}
Let $\bG$ be a locally compact quantum group.   $\bG$ has the {\it Haagerup property} if there exists a net of states $\{\mu_i \}$ on $C^\ast_u(\bG)$ such that $a_i := (\id \otimes \mu_i)(\mathcal{V}) \in L^\infty(\bG)$ satisfies the property that for every $c \in C_0(\bG)$ we have $\Vert a_i c - c \Vert_{L^\infty(\bG)} \rightarrow 0$.   
\end{dfn}

Let us comment on the existing examples. Firstly, we recall the following proposition from \cite[Proposition 5.2]{DawFimSkaWhi}. 

\begin{prop}\label{Prop=ImpliesHaagerup}
Let $\bG$ be a locally compact quantum group. If $\bG$ is coamenble, then $\hat{\bG}$ has the Haagerup property.
\end{prop}

The question whether amenability of $\bG$ implies Haagerup property remains open \cite[Remark 5.3]{DawFimSkaWhi}. 

Examples of quantum groups with the Haagerup property were  found amongst the amenable and coamenable quantum groups. These include quantum $E(2)$, quantum $ax+b$ and quantum $az+b$ and their duals. See \cite[Example 5.4]{DawFimSkaWhi} and references given there.   

Non-amenable examples so far come from discrete quantum groups. For the free orthogonal and free unitary quantum groups of Kac type, the Haagerup property was proved by Brannan \cite{Bran}. For   quantum reflexion groups, the Haagerup property was found by Lemeux \cite{Lem}. 

As a consequence of what we have proved  so far, we see that $\SUone$ is a non-compact, non-amenable quantum group that has the Haagerup property.

\begin{thm}\label{Thm=Haagerup}
Let $\bG = \SUone$. $\bG$ and $\hat{\bG}$ have both the Haagerup property.  
\end{thm}
\begin{proof}
$\hat{\bG}$ has the Haagerup property since $\bG$ is coamenable, c.f. Theorem \ref{Thm=Coamenability} and Proposition \ref{Prop=ImpliesHaagerup}.

Recall from \eqref{Prop=Analytic} that for $z \in [0,1)$ there exists a  unitary corepresentation of $\bG$, 
\[
V_z \in L^\infty(\bG) \otimes B(\cH_{z}).
\]  
Let $C_z$ be the the  C$^\ast$-algebra generated by $V_z$. It is the norm closure of the space spanned by slices $(\omega \otimes \id)(V_z)$ with $\omega \in L^1(\bG)$. 
By \cite[Proposition 5.3]{KusUniv}, there exists a non-degenerate $\ast$-homomorphism,
 \[
\pi_z: C^\ast_u(\bG) \rightarrow \mathcal{M}(C_z) \quad \textrm{ such that }  \quad V_z = (\id \otimes \pi_z)(\mathcal{V}).
\]
Here, $\mathcal{M}(C_z)$ is the multiplier algebra of $C_z$. Let the unit vector $f_z \in \cH_{z}$ be as in \eqref{Eqn=State} and set $a_z = (\id \otimes \omega_{f_z})(V_z)$ as in \eqref{Eqn=AzFunction}.  Using the fact that $C_z$ acts non-degenerately on $\cH_z$, it follows that $\omega_{f_z}$ is a state on $C_z$. Since $\pi_z$ is  non-degenerate, it preserves bounded approximate identities. For a state $\omega$ on a C$^\ast$-algebra $C$, it follows from Cohen's factorization theorem that $\Vert \omega \Vert = \omega(1) = \lim_j \omega(e_j)$, $\{ e_j\}$ being a bounded approximate identity of $C$ (here $\omega(1)$ is interpreted in the multiplier algebra of $C$). Hence, $\mu_z = \omega_{f_z} \circ \pi_z$ is a state on $C^\ast_u(\bG)$ and $a_z = (\id \otimes \mu_z)(\mathcal{V})$. Theorem \ref{Thm=ApproxSU}  proves that $\Vert a_z c - c\Vert_{L^\infty(\bG)} \rightarrow 1$ for every $c \in C_0(\bG)$ as $z \rightarrow 1$ (limit over the domain $[0,1)$).

\end{proof}

\appendix

\section{}

We prove the necessary technical results, which we have not found explicitly in the literature. Firstly, we have the following reformulation of a result of \cite{KusVae}. 

\begin{lem}[Proposition 8.9 of \cite{KusVae}]\label{Lem=ModularAppendix}
For every $t \in \mathbb{R}$ we have,
\[  
(\sigma_t \otimes \id)(W) = (\id \otimes \hat{\tau}_{-t})(W) (1 \otimes     \hat{\delta}^{-it} ) \quad {\rm and } \quad
(\id \otimes \hat{\sigma}_t)(W) =  (\delta^{it}  \otimes 1) (\tau_{-t} \otimes \id)(W).   
\]
\end{lem}
\begin{comment}
\begin{proof}
The first relation follows from,
\[
\begin{split}
(\sigma_t \otimes \id)(W) = & (\nabla^{it} \otimes 1) W (\nabla^{-it} \otimes 1)\\
= & (\hat{P}^{it} \hat{J} \hat{\delta}^{it} \hat{J} \otimes 1) W  (\hat{J} \hat{\delta}^{-it} \hat{J} \hat{P}^{-it}  \otimes 1)\\
= & (\hat{P}^{it} \hat{J} \hat{\delta}^{it}  \otimes J) W^\ast  ( \hat{\delta}^{-it} \hat{J} \hat{P}^{-it}  \otimes J)\\
= & (\hat{P}^{it} \hat{J}   \otimes J) W^\ast  (  \hat{J} \hat{P}^{-it}  \otimes \hat{\delta}^{it} J)\\
= & (\hat{P}^{it}    \otimes 1) W   (  \hat{P}^{-it}  \otimes J \hat{\delta}^{it} J)\\
= & (1   \otimes \nabla^{-it} ) W   (  1  \otimes  \nabla^{it} J \hat{\delta}^{it} J)\\
= & (1 \otimes \hat{\tau}_{-t})(W) (1 \otimes     \hat{\delta}^{-it} ).
\end{split}
\]
Here, the equalities follow from: Tomita-Takesaki theory; \cite[Proposition 8.9]{KusVae}; \cite[Corollary 2.2]{KusVaeII}; the fact that $\hat{W}^\ast (1 \otimes \hat{\delta}^{it}) \hat{W} = \hat{\Delta}(\hat{\delta}^{it}) = \hat{\delta}^{it} \otimes \hat{\delta}^{it}$ which implies that $(\hat{\delta}^{it} \otimes 1) W^\ast = W^\ast (\hat{\delta}^{it} \otimes \hat{\delta}^{it})$; again \cite[Corollary 2.2]{KusVaeII}; once more \cite[Corollary 2.2]{KusVaeII} (and the subsequent remark); the fact that $\hat{\tau}_t(x) = \nabla^{it} x \nabla^{-it}$ and \cite[Proposition 2.13 (15)]{KusVaeII}.

By the duality $\hat{W} = \Sigma W^\ast \Sigma$, the second formula follows.  
\end{proof}
\end{comment}

\begin{lem}\label{Lem=TomitaTransform}
Let $L^1(\bG)^\natural$ be the set of $\omega \in L^1(\bG)$ such that the following inclusions hold,
\[
(1) \:\: \omega \in L^1(\bG)^\sharp, \quad (2)  \:\:\omega \in \cI, \quad (3)\:\: \omega^\ast \in \cI.
\]
Let $L^1(\bG)^\flat$ be the set of all $\omega \in L^1(\bG)^\natural$ such that   $\lambda(\omega) \in \cT_{\hat{\varphi}}$ and for every $z \in \mathbb{C}$ there exists a functional $\omega_{[z]} \in L^1(\bG)^\natural$ with $\lambda(\omega_{[z]}) =  \hat{\sigma}_z(\lambda(\omega))$. Then, $L^1(\bG)^\flat$ is   a $\ast$-algebra,  dense in $L^1(\bG)$. Moreover, $\lambda(L^1(\bG)^\flat)$ is a $\sigma$-strong-$\ast$/norm core for $\hat{\Lambda}$. Furthermore, $\mathbb{C} \rightarrow L^1(\bG): z \mapsto \omega_{[z]}$ is analytic. 
\end{lem}
\begin{proof}
By \cite[Proposition 2.6]{KusVaeII} we see that $L^1(\bG)^\natural$ is a $\ast$-subalgebra of $L^1(\bG)^\sharp$, norm dense in $L^1(\bG)$. And $\lambda(L^1(\bG)^\natural)$ forms a $\sigma$-strong-$\ast$/norm core for $\hat{\Lambda}$. 

Define the norm continuous one-parameter group  $\rho$ acting on $L^1(\bG)$  by $\rho_t(\omega)(x) = \omega(\delta^{-it}\tau_{-t}(x))$, where $t \in \mathbb{R}$, see \cite[Notation 8.7]{KusVae}. For $\omega \in \cI$ we have $\rho_t(\omega) \in \cI$ and $\hat{\Lambda}(\lambda( \rho_t(\omega) )) = P^{-it} J \delta^{it} J \hat{\Lambda}(\lambda(\omega))$ for every $t \in \mathbb{R}$, c.f. \cite[Lemma 8.8]{KusVae}. By \cite[Proposition 7.12]{KusVae} we have $S(\delta^{it})^\ast = \delta^{it}$. Since $\tau_t$ and $R$ commute for every $t \in \mathbb{R}$ we have that $S = R \circ \tau_{-i/2}$ commutes with $\tau_t$. Hence, for $\omega \in L^1(\bG)^\sharp$ and $x \in D(S)$ we have that
\[
\begin{split}
& \langle S(x)^\ast, \rho_t(\omega) \rangle = \langle \delta^{-it} \tau_{-t} (S(x)^\ast), \omega \rangle = \langle \delta^{-it} S( \tau_{-t}(x) )^\ast, \omega \rangle \\ 
= &
\langle S(\delta^{-it} \tau_{-t}(x))^\ast , \omega \rangle 
 =  
\overline{ \langle \delta^{-it} \tau_{-t}(x), \omega^\ast \rangle   } = \overline{ \langle x, \rho_t(\omega^\ast) \rangle }
\end{split}
\]   
It follows that $\rho_t(\omega) \in L^1(\bG)^\sharp$ with $\rho_t(\omega)^\ast = \rho_t(\omega^\ast)$.

Now, let $\omega \in L^1(\bG)^\natural$. For each $t \in \mathbb{R}$, we have that $\rho_t(\omega) \in \cI$ and $\rho_t(\omega) \in L^1(\bG)^\sharp$. And also, $\rho_t(\omega)^\ast = \rho_t(\omega^\ast) \in \cI$ as $\omega^\ast \in \cI$. Thus $\rho_t(\omega) \in L^1(\bG)^\natural$.

We now use a standard smearing argument, which we take from \cite[Lemma 2.5]{KusVaeII}. Set,
\[
\omega(n, z) = \frac{n}{\sqrt{\pi}} \int_{-\infty}^\infty e^{-n^2 (t+z)^2} \rho_t(\omega) dt. 
\] 
From the closedness of the mapping $\omega \mapsto \hat{\Lambda}(\lambda(\omega))$ it follows that $\omega(n, z) \in \cI$ with,
\[
\hat{\Lambda}( \lambda( \omega(n, z) )) = \frac{n}{\sqrt{\pi}} \int_{-\infty}^\infty e^{-n^2 (t+z)^2} P^{-it} J \delta^{it} J \hat{\Lambda}( \lambda(\omega) ) dt. 
\]
From the previous paragraph, we see that $\omega(n, z) \in L^1(\bG)^\sharp$ with $\omega(n,z)^\ast = (\omega^\ast)(n, \overline{z})$. Hence, $\omega(n, z) \in L^1(\bG)^\natural$. Moreover, $\omega(n, 0)$ is analytic for $\rho$ and $\rho_z( \omega(n, 0) ) = \omega(n, -z)$.

Now, \cite[Proposition 8.9]{KusVae} shows that $\lambda(\rho_t(\omega)) = \hat{\sigma}_t( \lambda(\omega))$. Hence, from the smearing techniques, we find that $\lambda(\omega(n, 0))$ is analytic with respect to $\hat{\sigma}$. Also, $\omega(n, z) \in \cI$ and $\omega(n,z)^\ast \in \cI$ for all $z$, so that $\lambda(\omega(n, z))$ is in $\cT_{\hat{\varphi}}$. Since $\omega(n, 0) \rightarrow \omega$ in norm, we see that $L^1(\bG)^\flat$ is dense in $L^1(\bG)$. Moreover, we find that $\lambda(\omega(n, 0)) \rightarrow \lambda(\omega)$ in norm (and hence in the $\sigma$-strong-$\ast$ topology), and $\hat{\Lambda}(\lambda(\omega(n, 0) )) \rightarrow \hat{\Lambda}(\lambda(\omega))$ in norm. So indeed, $\lambda(L^1(\bG)^\flat)$ is a $\sigma$-strong-$\ast$/norm core for $\hat{\Lambda}$. That $L^1(\bG)^\flat$ is a $\ast$-algebra follows from \cite[Result 8.6]{KusVae} and the relation $(\omega \ast \theta)^\ast = \theta^\ast \ast \omega^\ast$.    

It remains to prove that {\it for every} $\omega \in L^1(\bG)^\flat$, the map $z \mapsto \omega_{[z]}$ is analytic. But for any $\theta \in L^1(\hat{\bG})$, the map $z \mapsto  \theta(\lambda(\omega_{[z]} )) = \theta(\hat{\sigma}_z(\lambda(\omega )))$ is analytic, which proves the claim since $\theta \circ \lambda$ with $\theta \in L^1(\hat{\bG})$ forms a set of separating functionals on $L^1(\bG)$.

\end{proof}

\begin{lem}[Remark 8.31 of \cite{KusVae}]\label{Lem=L2Thingy}
Let $\theta \in L^1(\bG)$ be such that $\lambda(\theta) \in L^2(\hat{\bG}) \cap L^\infty(\hat{\bG})$. Then, $\theta \in \cI$. 
\end{lem}
%\begin{proof} 

%Instead of carrying out tedious calculations as in the proof of Lemma \ref{Lem=TomitaTransform}, let us indicate the following short (but less self-contained) proof. Turn the triple of spaces  $L^1(\bG)$, $L^2(\bG)$, $L^\infty(\bG)$ into a compatible couple of Banach spaces as in \cite[Section 3]{CasLpf} and similarly for $\hat{\bG}$. Let, 
%\[
%\mathcal{F}_1: L^1(\bG) \rightarrow  L^\infty(\hat{\bG}), \quad \mathcal{F}_2: L^2(\bG) \rightarrow  L^2(\hat{\bG})
%\]
%be the compatible Fourier transforms \cite{CasLpf}.  Then, we see  $L^1(\bG) \ni \theta =  \hat{\mathcal{F}}_2(\mathcal{F}_1(\theta)) \in L^2(\bG)$ by the Fourier inversion theorem \cite[Corollary 5.5]{CasLpf}. Hence $\theta \in L^1(\bG) \cap L^2(\bG) = \cI$ by \cite[Theorem 3.3]{CasLpf}.  
%\end{proof}

\section{}\label{Sect=AppB}

Here, we prove the necessary results on convergences of basic hypergeometric series. 

\begin{lem}\label{Lem=LambdaFraction}
For every $k \in \mathbb{N} \cup \{ \infty\}$, we have the limit,
\[
\frac{(q/\lambda, q/\lambda; q^2)_k }{(1/\lambda^2; q^2)_k} \rightarrow 0,
\]
as $\lambda \rightarrow q$. 
\end{lem}
\begin{proof}
For $k=1$ and $k=2$ this is trivial. For $k \geq 3$ we need to be more careful, because the $q$-factorial in the denominator becomes 0 as $\lambda \rightarrow q$. However, we have,
\[
\begin{split}
& \frac{(q/\lambda, q/\lambda; q^2)_k }{(1/\lambda^2; q^2)_k}  = \prod_{i=0}^{k-1} \frac{(1- \frac{q}{\lambda} q^i)^2}{ 1 - \frac{1}{\lambda^2} q^i}, \\
  = & \frac{(1-q/\lambda)^2}{(1- q^2/\lambda^2)} \frac{(1-q^2/\lambda)^2}{(1-  1/\lambda^2)} \frac{(1-q^3/\lambda)^2}{(1- q/\lambda^2)} \prod_{i=3}^{k-1} \frac{(1- \frac{q}{\lambda} q^i)^2}{ 1 - \frac{1}{\lambda^2} q^i}\\
  = & \frac{ 1-q/\lambda }{1 + q^2/\lambda^2} \frac{(1-q^2/\lambda)^2}{(1-  1/\lambda^2)} \frac{(1-q^3/\lambda)^2}{(1- q/\lambda^2)} \prod_{i=3}^{k-1} \frac{(1- \frac{q}{\lambda} q^i)^2}{ 1 - \frac{1}{\lambda^2} q^i}.
\end{split}
\]
As $\lambda \rightarrow q$, this term goes to zero. 
\end{proof}

\begin{prop}\label{Prop=TwoPhiOneApproximation}
Let $a_z: I_q \rightarrow \mathbb{C}$ be the function defined in \eqref{Eqn=AzFunction}. Let $\cG_\alpha$ be the domain of Proposition \ref{Prop=Analytic}. Then $a_z(p_0) \rightarrow 1$ pointwise  as $z \rightarrow 1$ (limit in $\cG_\alpha$). Moreover, this convergence is uniform  on $I_q \cap [1, \infty)$.  
\end{prop}
\begin{proof}
Firstly, we record the following expression using \cite[Lemma 9.4]{GroKoeKus}. For every $p_0 \in I_q$, and $x = \mu(q^z) = \mu(\lambda)$ with $z \in \cG_\alpha$, we have,
\[
\begin{split}
a_z(p_0) = & \frac{1}{2} C(-x; 0, +, -; p_0,p_0, 0) +  \frac{1}{2} C(-x; 0, -, +; p_0,p_0, 0) \\
= & (-1)^{\frac{1}{2}(1-\sgn(p_0) )  } S(\lambda; p_0, p_0, 0) \\ 
& \: \times \: \left(\frac{A(-\lambda; 1, 0, \sgn(p_0), -\sgn(p_0)) }{A(-\lambda; 1, 0, +, -) } + \frac{A(-\lambda; 1, 0,-\sgn(p_0), \sgn(p_0)) }{A(-\lambda; 1, 0, -, +) } \right). 
\end{split}
\]
The fractions of the $A$-functions can be computed from \cite[Appendix B.6]{GroKoeKus} and one finds directly that,
\[
a_z(p_0) =  (-1)^{\frac{1}{2}(1-\sgn(p_0) )  } S(\lambda; p_0, p_0, 0).
\]
Next, we separate three cases. We prove that $a_z \rightarrow 1$  on each of the domains  $q^{-\mathbb{N} \cup \{0 \}}$, $q^{\mathbb{N}}$, and $ - q^{\mathbb{N}}$ and that the convergence is uniform on the first one. Reason for the separation of cases is that we need to consider analytic extensions of basic hypergeometric $_2 \varphi _1$-series using the transformation formula \cite[Eqn. (III.32)]{GasRah} on different domains. 

\vspace{0.3cm}

\noindent {\bf Case 1:} The domain $p_0 \in q^{-\mathbb{N} \cup \{0 \}}$. In this case, using exactly the same computations as in \eqref{Eqn=CoamenabilityCompI} and \eqref{Eqn=CoamenabilityCompII} for $m= 0$, we see that, 
\[
\begin{split}
a_z(p_0) = &  \: _2 \! \varphi_1\left( \begin{array}{c} q/\lambda,  \lambda q \\ q^2 \end{array};q^2, -q^2/\kappa(p_0) \right).
\end{split}
\]
In case $z \rightarrow 1$ we see that $\lambda \rightarrow q$ and $a_z(p_0) \rightarrow 1$ uniformly for  $p_0 \in  q^{-\mathbb{N} \cup \{0 \}}$ (that the convergence is uniform follows directly from the defining power series of the basic hypergeometric series, which terminates for $\lambda = q$ as a constant function).

\vspace{0.3cm}

\noindent {\bf Case 2:}  The domain $p_0 \in q^{\mathbb{N}}$. In this case, we find using the expression \cite[Lemma 9.1]{GroKoeKus}, the $\theta$-product formula and the transformation formula \cite[Eqn. (III.32)]{GasRah},
\[
\begin{split}
& a_z(p_0)   \\
= 
& S(\lambda, p_0, p_0, 0) \\
= & p_0^2 \nu(p_0)^2 c_q^2 (-p_0^2; q^2)_\infty (q^2, -q^2/p_0^2; q^2)_\infty (q^2; q^2)_\infty
 \: _2 \! \varphi_1\left( \begin{array}{c} q/\lambda,  \lambda q \\ q^2 \end{array};q^2, -q^2/\kappa(p_0) \right) \\
= &   \: _2 \! \varphi_1\left( \begin{array}{c} q/\lambda,  \lambda q \\ q^2 \end{array};q^2, -q^2/\kappa(p_0) \right)\\
= &  
\frac{(q \lambda, q\lambda, -q^3/\lambda\kappa(p_0), - \lambda \kappa(p_0)/q; q^2 )_\infty  }{(q^2, \lambda^2, -q^2/\kappa(p_0), -\kappa(p_0); q^2 )_\infty  } \: _2 \! \varphi_1\left( \begin{array}{c} q/\lambda,    q/\lambda \\ q^2/\lambda^2 \end{array};q^2, -\kappa(p_0) \right) \\
& \: + \: 
\frac{(q /\lambda, q/\lambda, -q^3 \lambda / \kappa(p_0), - \kappa(p_0)/q \lambda; q^2 )_\infty  }{(q^2, 1/ \lambda^2, -q^2/\kappa(p_0), -\kappa(p_0); q^2 )_\infty  } \: _2 \! \varphi_1\left( \begin{array}{c} q \lambda,    q \lambda \\ q^2 \lambda^2 \end{array};q^2, -\kappa(p_0) \right)\\
\end{split}
\]
By Lemma \ref{Lem=LambdaFraction} we see that the second summand goes to 0 if $\lambda \rightarrow q$.    Hence, we find formally that,
\begin{equation}\label{Eqn=BasicHypComp}
\begin{split}
&\lim_{z \rightarrow 1} a_z(p_0)   \\
= & \lim_{\lambda \rightarrow q}  
\frac{(q \lambda, q\lambda, -q^3/\lambda\kappa(p_0), - \lambda \kappa(p_0)/q; q^2 )_\infty  }{(q^2, \lambda^2, -q^2/\kappa(p_0), -\kappa(p_0); q^2 )_\infty  } \\
& \: \times \: \: _2 \! \varphi_1\left( \begin{array}{c} q/\lambda,    q/\lambda \\ q^2/\lambda^2 \end{array};q^2, -\kappa(p_0) \right),  \\  
= & \lim_{\lambda \rightarrow q}   \: _2 \! \varphi_1\left( \begin{array}{c} q/\lambda,    q/\lambda \\ q^2/\lambda^2 \end{array};q^2, -\kappa(p_0) \right),  \\ 
\end{split}
\end{equation}
however we have to justify the existence of the (latter) limit. In order to do this, we have for $k \in \mathbb{N}$,
\[
\frac{(q/\lambda, q/\lambda ; q^2)_k }{ (q^2/\lambda^2; q^2)_k } = \frac{ (1-\frac{q}{\lambda}  )^2 }{1- \frac{q^2}{\lambda^2}   }  \prod_{i=1}^{k-1} \frac{ (1-\frac{q}{\lambda} q^i)^2 }{1- \frac{q^2}{\lambda^2} q^i } =  
 \frac{  1-\frac{q}{\lambda}   }{1 + \frac{q}{\lambda}   }  \prod_{i=1}^{k-1} \frac{ (1-\frac{q}{\lambda} q^i)^2 }{1- \frac{q^2}{\lambda^2} q^i }. 
\]
As $\lambda \rightarrow q$ this expression goes to zero. This means that the coefficients of the basic hypergeometric series in  \eqref{Eqn=BasicHypComp} go to 0 as $\lambda \rightarrow q$, resulting in the constant function one. Thus $\lim_{z \rightarrow 1} a_z(p_0) = 1$ pointwise for $p_0 \in q^{\mathbb{N}}$.

\vspace{0.3cm}

\noindent {\bf Case 3:}  The domain $p_0 \in -q^{\mathbb{N}}$. We find that using \cite[Lemma 9.1]{GroKoeKus} and the transformation formula \cite[Eqn. (III.32)]{GasRah} that,
\begin{equation}\label{Eqn=BasicSeriesCompII}
\begin{split}
& - a_z(p_0) \\
= & p_0^2 \nu(p_0)^2 c_q^2 (- \kappa(p_0) ; q^2)_\infty 
\frac{( q^2, -q^2/\kappa(p_0), -\lambda q^3/p_0^2, - p_0^2/q\lambda; q^2 )_\infty}{( p_0^2/q\lambda, \lambda q^3/p_0^2; q^2 )_\infty } (q^2; q^2)_\infty\\
  & \: \times \: \: _2 \! \varphi_1\left( \begin{array}{c} q/\lambda,    q \lambda \\ q^2  \end{array};q^2, -q^2/\kappa(p_0) \right) \\
= &  p_0^2 \nu(p_0)^2 c_q^2 (- \kappa(p_0) ; q^2)_\infty  (q^2; q^2)_\infty^2 \frac{(q^2/p_0^2, -\lambda q^3/p_0^2, -p_0^2/q\lambda; q^2 )_\infty }{ (p_0^2/q\lambda, \lambda q^3/p_0^2; q^2)_\infty  } \\
  & \: \times \:  \left( 
\frac{( \lambda q, \lambda q, -q^3/\lambda \kappa(p_0), -\lambda \kappa(p_0) /q; q^2 )_\infty}{( q^2,  \lambda^2, -q^2/ \kappa(p_0), -  \kappa(p_0) ; q^2 )_\infty } \: _2 \! \varphi_1\left( \begin{array}{c} q/\lambda,    q /\lambda \\ q^2/\lambda^2  \end{array};q^2, -\kappa(p_0) \right)\right. \\
& \: + \: \left. 
\frac{( q/\lambda , q/\lambda , -q^3\lambda/ \kappa(p_0), -\kappa(p_0) / \lambda q; q^2 )_\infty}{( q^2,  1/\lambda^2, -q^2/ \kappa(p_0), -  \kappa(p_0) ; q^2 )_\infty } \: _2 \! \varphi_1\left( \begin{array}{c} q \lambda,    q  \lambda \\ q^2 \lambda^2  \end{array};q^2, -\kappa(p_0) \right)
\right) 
\end{split}
\end{equation}
As in {\bf Case 2} it follows from Lemma \ref{Lem=LambdaFraction} that the second summand within the large brackets of \eqref{Eqn=BasicSeriesCompII} tends to zero as $\lambda \rightarrow q$.  Therefore, using that $p_0 := -q^k$ is negative, and the $\theta$-product identity various times,
\[
\begin{split}
& \lim_{z \rightarrow 1} - a_z(p_0) \\
= & \lim_{\lambda \rightarrow q}  p_0^2 \nu(p_0)^2 c_q^2 (- \kappa(p_0) ; q^2)_\infty  (q^2; q^2)_\infty^2 \frac{(q^2/p_0^2, -\lambda q^3/p_0^2, -p_0^2/q\lambda; q^2 )_\infty }{ (p_0^2/q\lambda, \lambda q^3/p_0^2; q^2)_\infty  } \\
  & \: \times \:  
\frac{( \lambda q, \lambda q, -q^3/\lambda \kappa(p_0), -\lambda \kappa(p_0) /q; q^2 )_\infty}{( q^2,  \lambda^2, -q^2/ \kappa(p_0), -  \kappa(p_0) ; q^2 )_\infty } \: _2 \! \varphi_1\left( \begin{array}{c} q/\lambda,    q /\lambda \\ q^2/\lambda^2  \end{array};q^2, -\kappa(p_0) \right)  \\
= & \lim_{\lambda \rightarrow q}  p_0^2 \nu(p_0)^2 c_q^2 (p_0^2 ; q^2)_\infty  (q^2; q^2)_\infty^2 (-q^4/p_0^2, -p_0^2/q^2; q^2)_\infty   \frac{(\lambda q, \lambda q;q^2)_\infty  }{ (q^2, \lambda^2;q^2)_\infty } \\
  & \: \times \: \frac{(q^3/\lambda p_0^2, \lambda p_0^2/q; q^2)_\infty }{(q^3\lambda/ p_0^2,  p_0^2/q \lambda ; q^2)_\infty }  \frac{1}{(p_0^2; q^2)_\infty} \: _2 \! \varphi_1\left( \begin{array}{c} q/\lambda,    q /\lambda \\ q^2/\lambda^2  \end{array};q^2, -\kappa(p_0) \right)  \\
= & \lim_{\lambda \rightarrow q}  p_0^2 \nu(p_0)^2 c_q^2 (p_0^2 ; q^2)_\infty  (q^2; q^2)_\infty (-q^4/p_0^2, -p_0^2/q^2; q^2)_\infty \frac{(q^2, q^2; q^2)_\infty}{ (q^2;q^2)_\infty}  \\
& \: \times \:  \frac{1-\frac{q^2}{p_0^2}}{1-\frac{p_0^2}{q^2}} \frac{1}{(p_0^2; q^2)_\infty} \: _2 \! \varphi_1\left( \begin{array}{c} q/\lambda,    q /\lambda \\ q^2/\lambda^2  \end{array};q^2, -\kappa(p_0) \right)\\
= &   p_0^2 \nu(p_0)^2 c_q^2   (q^2; q^2)_\infty (-q^4/p_0^2, -p_0^2/q^2; q^2)_\infty \frac{(q^2, q^2; q^2)_\infty}{ (q^2;q^2)_\infty} \frac{q^2}{p_0^2} \frac{p_0^2 - q^2}{q^2 - p_0^2} \\
= & - q^2  q^{(k-1)(k-2)} c_q^2 (q^2; q^2)_\infty^2 (-1, -q^2; q^2)_\infty q^{-(k-1)(k-2)}  \\
= & -c_q^2 (q^2; q^2)_\infty^2 (-1, -q^2; q^2)_\infty q^2\\
= & -1. 
\end{split}
\]

\end{proof}

{\bf Acknowledgements.} The author thanks Wolter Groenevelt for useful correspondence on the corepresentation spectrum of extended quantum $SU(1,1)$. The author  thanks Erik Koelink for several indispensable discussions. The author also thanks the IMAPP institute at the Radboud Universiteit Nijmegen for their hospitality; the final part of this paper was completed there. We thank Yuki Arano for pointing out Remark \ref{Rmk=Reflexion}. We thank the anonymous referee for several improvements of the manuscript.

\end{document}